\documentclass[pdflatex,sn-mathphys-num]{sn-jnl}

\usepackage{graphicx}%
\usepackage{multirow}%
\usepackage{amsmath,amssymb,amsfonts}%
\usepackage{amsthm}%
\usepackage{mathrsfs}%
\usepackage[title]{appendix}%
\usepackage{xcolor}%
\usepackage{textcomp}%
\usepackage{manyfoot}%
\usepackage{booktabs}%
\usepackage{algorithm}%
\usepackage{algorithmicx}%
\usepackage{algpseudocode}%
\usepackage{listings}%
\usepackage[english]{babel}
\usepackage{amscd}
\usepackage{amsmath,amsthm,amssymb,amsfonts}
\usepackage{mathtools}
\usepackage{euler}
\usepackage{MnSymbol}
\usepackage{tikz-cd}
\usepackage{enumitem}
\usepackage{hyperref}

\theoremstyle{thmstyleone}%
\newtheorem{theorem}{Theorem}

\newtheorem{proposition}[theorem]{Proposition}

\newtheorem{lemma}[theorem]{Lemma}
\newtheorem{remark}[theorem]{Remark}

\def\AA{\mathbb{A}}

\def\EE{\mathbb{E}}

\def\I{\mathcal{I}}
\def\J{\mathcal{J}}
\def\KK{\mathbb{K}}

\def\mf{\mathfrak{m}}

\def\NN{\mathbb{N}}

\def\O{\mathcal{O}}
\def\PP{\mathbb{P}}
\def\sU{\mathscr{U}}

\def\V{\mathcal{V}}

\DeclareMathOperator*{\Ass}{Ass}

\DeclareMathOperator*{\height}{height}

\DeclareMathOperator*{\HF}{HF}

\DeclareMathOperator*{\indeg}{indeg}
\DeclareMathOperator*{\Proj}{Proj}
\DeclareMathOperator*{\rk}{rank}
\DeclareMathOperator*{\reg}{reg}

\DeclareMathOperator*{\sat}{sat}
\DeclareMathOperator*{\Span}{span}
\DeclareMathOperator*{\Spec}{Spec}
\DeclareMathOperator*{\Tor}{Tor}

\begin{document}

\author{Xin Hong$^1$}
\author{Yi-Song Huang$^1$}
\author{Manolis Tsakiris$^{1}$}

\title{A sublinear bound for the regularity of subspace arrangements}

\abstract{Let $X$ be the union of $k$ generic linear spaces in $\PP^n$ of codimension at least $2$. We prove that the Castelnuovo-Mumford regularity of the vanishing ideal of $X$ grows in order not more than $\sqrt{k}$, and that this is an order-tight bound when the codimension equals $2$.}


\maketitle

\long\def\thefootnote{} 
\footnotetext{\noindent $^1$ Academy of Mathematics and Systems Science, Chinese Academy of Sciences, No. 55 Zhongguancun East Road, Beijing, 100190, China.}
\footnotetext{\noindent Emails: \texttt{hongxin@amss.ac.cn}, \texttt{huangyisong@amss.ac.cn}, \texttt{manolis@amss.ac.cn}}
\footnotetext{\noindent Work supported by the National Key R\&D Program of China under Grant No. 2023YFA1009402.}
\long\def\thefootnote{\arabic{footnote}} 

\keywords{Castelnuovo-Mumford regularity, Hilbert functions, subspace arrangements, flat degenerations}




\section{Introduction}\label{sec1}
In their classical masterpiece \cite{hartshorne1982droites}, Hartshorne and Hirschowitz described the Hilbert function of the union of $k$ generic lines in $\PP^n$.
Since then, understanding the Hilbert function of the union $X$ of $k$ higher dimensional generic linear spaces has been largely open (\cite{aladpoosh2021hartshorne}, \cite{carlini2010bipolynomial}). Besides its own intrinsic interest, an important intermediate step in this direction is the study of the Castelnuovo-Mumford regularity of the ideal $I_X$ of such a union.
Recently, Conca and the last author proved in \cite{conca2024castelnuovo} that $\reg(I_X) \le k - \lfloor \frac{k}{2n-1} \rfloor$ for generic subspaces of codimension at least $2$, improving the classical bound $\reg(I_X) \le k$ of Derksen and Sidman \cite{derksen2002sharp}.
Yet, numerical calculations in small examples have been suggesting that the regularity grows sublinearly with respect to $k$.
In this manuscript we are able to prove: 

\begin{theorem} \label{thm:main}
Let $X$ be the reduced union of $k$ generic linear spaces in $\PP^n$ of codimension at least $2$.
Let $j$ be the smallest integer such that $k \le \binom{j+1}{2}$.
Then $\reg(I_X) \le (n-1)j $,
where $I_X$ is the homogeneous saturated ideal that defines $X$.
\end{theorem}

Theorem \ref{thm:main} shows that $\reg(I_X)$ is bounded in order by the square root of the number $k$ of subspaces. Since the regularity of $I_X$ is an upper bound for the initial degree of $I_X$, this is indeed the correct order for codimension-$2$ subspaces, due to the following fact.

\begin{proposition} \label{prp:lower-bound}
Let $X$ be the reduced union of $k$ codimension $2$ generic linear spaces in $\PP^n$. Denote by $\indeg(I_X)$ the smallest integer $d$ such that $I_X$ is non-zero at degree $d$. Then 
\begin{align*}
\indeg(I_{X}) \ge 
\begin{cases}
\lfloor\frac{-3+\sqrt{1+8k}}{2}\rfloor+1, & n=2 \\
\lfloor \frac{-5+\sqrt{1+24k}}{2} \rfloor+1, & n=3 \\
\max \left\{ 3\sqrt{k}-4, \lfloor \frac{-5+\sqrt{1+24k}}{2} \rfloor+1 \right\}, & n\ge 4
\end{cases}
\end{align*} with the first two inequalities being equalities.
\end{proposition}

The proof of Theorem \ref{thm:main} is delivered in \S \ref{section:Proof-Main}, Proposition \ref{prp:lower-bound} is proved in \S \ref{section:lower-bound}, while \S \ref{section:Preliminaries} is concerned with preliminaries. We are grateful to Aldo Conca for pointing to us Lemma \ref{lem:comparison-I-J}, for his comments on this work and for other related inspiring discussions.

\section{Preliminaries} \label{section:Preliminaries}

In this section we concisely organize for ease of reference notation and existing results that we will latter need.  

\subsection{Basic notation}

With $\KK$ an infinite field, $S = \KK[x_0,...,x_n]$ will be the polynomial ring of dimension $n+1$ equipped with the standard grading. We denote by $\mathfrak{m}$ the maximal homogeneous ideal of $S$. If $M$ is a graded $S$-module and $d$ is a positive integer, $[M]_d$ will denote the $\KK$-vector space of degree $d$ homogeneous elements of $M$; in particular $[S]_1$ is the $(n+1)$-dimensional vector space of linear forms of $S$. When $k$ is a positive integer, $[k]$ indicates the set $\{1,\dots,k\}$. 

\subsection{Regularity}

The Castelnuovo-Mumford regularity $\reg(M)$ of a finitely generated graded $S$-module $M$ is defined as the largest integer $q+d$ such that the degree $d$ component of the local cohomology module $H_{\mf}^q(M)$ is non-zero. 

\begin{proposition}[Theorem 2.1 in \cite{derksen2002sharp}] \label{prp:Derksen-Sidman}
Let $I_1,\dots,I_k$ be ideals of $S$ generated by linear forms. Then 
$$\reg \left(\cap_{i \in [k]} I_i \right) \le k.$$
\end{proposition}

\subsection{Comparison results on regularity}

With $c \in [n+1]$, let $L_c$ be the ideal of $S$ generated by $c$ generic linear forms.  

\begin{proposition}[Proposition 22 in \cite{conca2024castelnuovo}]  \label{prp:Lc vs L(c+1)}
For every non-zero saturated homogeneous ideal $J$, and every $c \in [n]$, we have 
$$ \reg(J \cap L_c) \ge \reg(J \cap L_{c+1}).$$
\end{proposition}

\begin{proposition} [weak form of Theorem 3 in \cite{conca2024castelnuovo}] \label{prp:J vs J cap Lc}
For every non-zero saturated homogeneous ideal $J$, and every $c \in [n+1]$, we have 
$$ \reg(J) \le  \reg(J \cap L_c).$$
\end{proposition}

\subsection{Rank functions of subspace arrangements} \label{subsection:rank functions}

A collection $\V = \{V_i : \, i \in [k]\}$ of linear subspaces of $[S]_1$, where $k$ is a positive integer, gives rise to a \emph{rank function} $\rk_{\V}: 2^{k} \rightarrow \NN$ that takes any subset $A \subseteq [k]$ to 
$$ {{\rk}_\V}(A) := \dim_{\KK} \sum_{i \in A} V_i.$$ 

Taking the \emph{Dilworth truncation} \cite{Fujishige} of $\rk_\V$ gives another rank function $\rk^*_{\V}: 2^{k} \rightarrow \NN$ that takes the empty set to zero and any non-empty $A \subseteq [k]$ to
$$ {\rk}_{\V}^*(A) := \min_{A = \bigsqcup_{i \in [p]} A_i} \left[ \sum_{i \in [p]} {\rk}_{\V}(A_i) - p \right],$$ where the minimum extends over all possible partitions of $A$.

\subsection{Betti numbers}

The rank function ${\rk}^*_{\V}$ induces a discrete polymatroid $P(V)^* \subset \NN^k$, which is the set 
$$ P(V)^* = \left\{x \in \NN^k: \, \sum_{i \in A} x_i \le {\rk}_\V^*(A), \, \, \, \text{for every non-empty} \, \, \, A \subseteq [k] \right\}.$$ 

\noindent Now for every $i \in [k]$ let $I_i$ be the ideal of $S$ generated by the elements in $V_i$. Set $J = \prod_{i \in [k]} I_i$. The polymatroid $P(\V)^*$ governs the Betti numbers of $J$:

\begin{proposition}[Corollary 2.4(iii) in \cite{conca2019resolution}] \label{prp:Betti}
The $i$-th Betti number of $J$ is given by 
$$\beta_i(J) = \sum_{x \in P(\V)^*} \binom{|x|}{i},$$ where 
$|x| = \sum_{i \in [k]} x_i$ for $x \in \NN^k$.
\end{proposition}

\subsection{Primary decomposition and associated primes}

A primary decomposition of $J$ was proved by Conca \& Herzog:

\begin{proposition}[Lemma 3.2 in \cite{conca2003castelnuovo}] \label{prp:primary decomposition}
A possibly redundant primary decomposition of the product ideal $J = \prod_{i \in [k]} I_i$ is given by 
$$J = \bigcap_{\emptyset \neq A \subseteq [k]} \left(\sum_{i \in A} I_i \right)^{|A|}.$$
\end{proposition}

More recently, Conca \& Tsakiris showed that the polymatroid $P(\V)^*$ also governs the associated primes of $J$:

\begin{proposition}[Proposition 3.1 in \cite{conca2019resolution}] \label{prp:associated primes}
For $A \subseteq [k]$, we have that $\sum_{i \in A} I_i \in \Ass(S/J)$ if and only if 
\begin{enumerate}[label=(\roman*)]
\item ${\rk}_{\V}(A) < {\rk}_{\V}(B)$ for any $B \supsetneq A$, and
\item ${\rk}_{\V}^*(A) = {\rk}_{\V}(A)-1$.
\end{enumerate}
\end{proposition}

\begin{remark} \label{rem:irredundant-primary-decomposition}
Combined, Propositions \ref{prp:primary decomposition} and \ref{prp:associated primes} provide an irredundant primary decomposition for the product ideal $J$.
\end{remark}

When the subspace arrangement $\V$ is \emph{linearly general}, that is when for every $A \subseteq [k]$ we have 
\begin{align*}
\dim_{\KK} \sum_{i \in A} V_i = \min\left\{n+1, \sum_{i \in A} \dim_{\KK} V_i \right\},
\end{align*} the primary decomposition of the product ideal $J$ becomes particularly simple:

\begin{proposition}[Proposition 3.4 in \cite{conca2003castelnuovo}] \label{prp:primary-decomposition-linearly-general}
If the $V_i, \, i \in [k],$ form a linearly general subspace arrangement, then a primary decomposition of the product ideal $J$ is given by the following formula:
\begin{align*}
J = I_1\cdots I_k = I_1 \cap \cdots \cap I_k \cap \mathfrak{m}^k.
\end{align*}
\end{proposition}

\subsection{Star configurations} \label{subsection:star-configurations}

Let $f_1,\dots,f_j$ be linear forms of $S$, such that the arrangement of dimension $1$ linear subspaces $\Span(f_i)$ with $i \in [j]$ of $[S]_1$ is linearly general. With $c \in [j]$ a positive integer, these linear forms induce a special subspace arrangement of $[S]_1$, where each subspace corresponds to a $c$-tuple of integers $1 \le \alpha_1 < \cdots < \alpha_c \le j$ and is given by $\Span(f_{\alpha_1},\dots,f_{\alpha_c})$. The projective scheme in $\PP^n$ defined by the ideal
\begin{align*}
A_{c,j} = \bigcap_{1 \le \alpha_1 < \cdots < \alpha_c \le j} (f_{\alpha_1},\dots,f_{\alpha_c}) 
\end{align*} is referred to as a codimension-$c$ \emph{star configuration} of $j$ hyperplanes. 

\begin{proposition}[Proposition 2.9 in \cite{StarConfigurations}] \label{prp:star-configurations}
The ideal $A_{c,j}$ of a star configuration satisfies the following properties:
\begin{enumerate}[label=(\roman*)]
\item $S/A_{c,j}$ is Cohen-Macaulay of dimension $n+1-c$.
\item $\reg(A_{c,j}) = j-c+1$.
\item $A_{c,j}$ is generated by all products $f_{\alpha_1}\cdots f_{\alpha_{j-c+1}}$ for $1 \le \alpha_1 < \cdots < \alpha_{j-c+1} \le j$.
\item The $h$-vector of $S/A_{c,j}$ is 
\begin{align*}
\left(1, \, {c \choose c-1}, \, {c+1 \choose c-1}, \dots, {j-1 \choose c-1} \right).
\end{align*}
\end{enumerate}
\end{proposition}

\section{Proof of Theorem \ref{thm:main}} \label{section:Proof-Main}

The proof of Theorem \ref{thm:main} follows from the synthesis of several ingredients that we next develop. 

\subsection{Reduction to codimension $2$ and $k=\binom{j+1}{2}$}

We may write $I_X = \cap_{i \in [k]} I_i$, where $I_i$ is an ideal of $S$ generated by at least two generic linear forms. Pick any two generic linear forms in $I_i$ and let $I_i'$ be the ideal generated by them. Inductive application of Proposition \ref{prp:Lc vs L(c+1)} furnishes 
$$ \reg(I_X) \le  \reg\left(\cap_{i \in [k]} I_i' \right).$$
Set $k' = {j +1 \choose 2}$. Then Proposition \ref{prp:J vs J cap Lc} implies that 
\begin{align*} \reg\left(\cap_{i \in [k]} I_i' \right) \le \reg\left(I_{X'}\right) \end{align*}
where $X'$ is the union of $k'$ codimension-$2$ generic linear spaces in $\PP^n$. Combining the above two inequalities, we see that it suffices to prove the theorem for $X'$. Thus the rest of the proof is devoted in proving:

\begin{theorem} \label{thm:codimension2}
Let $X$ be the reduced union of ${j+1 \choose 2}$ generic linear spaces of codimension $2$ in $\PP^n$. Then $\reg(I_X) \le (n-1)j $.
\end{theorem}
 
\subsection{Induction hypothesis} \label{subsection:induction-hypothesis}

We prove Theorem \ref{thm:codimension2} by induction on $n \ge 2$. The base of the induction corresponds to $n=2$, and is taken care of by the next lemma (showing that in fact the bound is tight).

\begin{lemma} \label{lem:regularity of k points}
The ideal of $k = {j+1 \choose 2}$ generic points in $\PP_{\KK}^2$ has regularity $j$ and a linear resolution.
\end{lemma}
\begin{proof}
The Hilbert function of the ideal $I_X$ of $k = {j+1 \choose 2}$ generic points in $\PP_{\KK}^2$ is known to be 
\begin{align*}
\HF(I_X,d) = \max \left\{0, {d + 2 \choose 2} - {j+1 \choose 2} \right\}.
\end{align*} From this formula we read that $I_X$ vanishes at degrees $\le j-1$, and that the regularity of the Hilbert function of $I_X$ (that is, the smallest degree from which the Hilbert function agrees with the Hilbert polynomial thereafter) is
\begin{align*}
{\reg}_{\text{HF}} (I_X) = j-1.
\end{align*} On the other hand, as $I_X$ is saturated and $S/I_X$ has Krull dimension $1$, Serre's formula reads
\begin{align*}
\HF(S/I_X,d) - P_{S/I_X}(d) = - \dim_{\KK} \left[H_m^1(S/I_X) \right]_d,
\end{align*} where $P_{S/I_X}$ is the Hilbert polynomial of $S/I_X$. We infer that 
\begin{align*}
{\reg}_{\text{HF}} (S/I_X) = \reg(S/I_X).
\end{align*} As the regularity of the Hilbert function of $S/I_X$ agrees with that of $I_X$, we conclude that 
\begin{align*}
\reg(I_X) = {\reg}_{\text{HF}} (I_X)+1 = j.
\end{align*} As $I_X$ has regularity $j$ and vanishes at degrees less than $j$, it must have a linear resolution.
\end{proof}

For $n \ge 3$, we induct on $j$. If $j \le 2n-3$, then 
\begin{align*}
k = {j+1 \choose 2}  = \frac{j+1}{2} j \le \frac{2n-2}{2} j = (n-1)j,
\end{align*} and the statement follows from Proposition \ref{prp:Derksen-Sidman}. Consequently, we will hereafter assume that $j \ge 2n-2$ and $n \ge 3$, with the theorem being true for smaller values of $j$ and $n$.  

\subsection{The product ideal $J_{\xi,t}$}

Recall $S = \KK[x_0,\dots,x_n]$. For integers $0 \le p \le n $ and $1 \le q \le k$, we let $\xi_{pq}, \, \xi_{pq}'$ be algebraically independent elements over the field of fractions of $S$, and define a field
\begin{align*}
\EE = \left(\KK(\xi_{pq}, \, \xi_{pq}': \, 0 \le p \le n, \, 1 \le q \le k)\right)^a
\end{align*} as the algebraic closure of the field of rational functions in the $\xi_{pq}$'s and the $\xi_{pq}'$'s over $\KK$. We then define a new polynomial ring 
\begin{align*}
R = S \otimes_{\KK} \EE = \EE[x_0,\dots,x_n].
\end{align*} With $t$ transcendental over the field of fractions of $R$, we consider the extended polynomial ring $R[t]$; this is still graded with $\deg(x_i) = 1$, $\deg(t) = 0$ and $\deg(\xi) = 0$ for every $\xi \in \EE$. Inside $R[t]$ we define linear forms 
\begin{align*}
\ell_i & = \xi_{i0} x_0 + \xi_{i1} x_1 + \cdots + \xi_{in} x_n \\
\ell_i' & = \xi_{i0}' x_0 + \xi_{i1}' x_1 + \cdots + \xi_{in}' x_n.
\end{align*} With $j \in [k]$, a central object in the proof will be the ideal of $R[t]$ \begin{align*}
J_{\xi,t} = (\ell_1,x_0+t\ell_1')\cdots (\ell_j,x_0+t\ell_j')(\ell_{j+1},\ell_{j+1}')\cdots(\ell_k,\ell_k').
\end{align*}

\subsection{The subspace arrangement $V_{\xi,t}$}

With $\EE(t)$ the field of rational functions in $t$ over $\EE$, we consider 
the localization of $R[t]$ 
\begin{align*}
R(t) = R[t] \otimes_\EE \EE(t) = \EE(t)[x_0,\dots,x_n],
 \end{align*} which is a standard graded polynomial ring over the field $\EE(t)$. Inside the $(n+1)$-dimensional $\EE(t)$-vector space of linear forms of $R(t)$, we consider the subspace arrangement 
\begin{align*}
\V_{\xi, t}=\big(&\Span(\ell_1,x_0+t\ell_1'),\dots,\Span(\ell_j,x_0+t\ell_j'), \\
&\Span(\ell_{j+1},\ell_{j+1}'), \dots, \Span(\ell_k,\ell_k') \big).
\end{align*} 

\begin{lemma} \label{lem:rank*-Vc}
The subspace arrangement $\V_{\xi,t}$ is linearly general, and for every $A \subseteq [k]$ $${\rk}_{\V_{\xi,t}}^*(A) = \min\{|A|,n\}.$$
\end{lemma}
\begin{proof}
For any $A \subseteq [k]$, write $A = A_{\le j} \sqcup A_{>j}$, where $A_{\le j}$ is the set of elements in $A$ less than or equal to $j$ and $A_{>j}$ the elements in $A$ bigger than $j$. We construct an $(n+1) \times (2|A|)$ matrix $M$ with elements in $\EE(t)$, by letting its columns contain the coefficients of the linear forms $\ell_i$ and $x_0 + t \ell_i'$ for $i \in A_{\le j}$ and $\ell_i$ and $\ell_i'$ for $i \in A_{>j}$. Note that the entries of the matrix $M$ are the elements (up to a reordering)
\begin{align*}
\xi_{iv} &:  \, \, \,  i \in A, \, 0 \le v \le n \\
 \xi_{iv}'&:  \, \, \, i \in A_{>j}, \, 0 \le v \le n \\
 1+t \xi_{i0}' &:  \, \, \, i \in A_{\le j} \\
 t \xi_{iv}' &:  \, \, \, i \in A_{\le j}, \,  1 \le v \le n.
\end{align*} Since the $\xi_{iv}$'s, $\xi_{iv}'$'s and $t$ are jointly algebraically independent over $\KK$, one easily checks that so are the elements listed above. Hence every maximal minor of $M$ is a non-zero polynomial of $$\KK \left[\{\xi_{iv}, \, \xi_{iv}'\}_{i \in [k], \, 0 \le v \le n}, t\right].$$ 

\noindent In particular, the rank of $M$ over the field $\EE(t)$ is $\min\{n+1,2|A|\}$. In other words,  
\begin{align*}
{\rk}_{\V_{\xi,t}}(A) = \min\{n+1,2|A|\};
\end{align*} that is, $\V_{\xi,t}$ is linearly general.

We first argue that ${\rk}_{\V_{\xi,t}}^*(A) \le \min\{|A|,n\}.$ To see this, we recall the definition of ${\rk}_{\V_{\xi,t}}^*(A)$, as the minimum value of 
\begin{align*}
\sum_{i \in [p]} {\rk}_{\V_{\xi,t}}(A_i)-p
\end{align*} over all partitions $A = \bigsqcup_{i \in [p]} A_i$ of $A$. Selecting the trivial partition with $p=1$ and $A_1 = A$, we have 
\begin{align*}
{\rk}_{\V_{\xi,t}}^*(A) \le {\rk}_{\V_{\xi,t}}(A)-1 \le (n+1)-1 = n.
\end{align*} On the other extreme, the partition of $A$ with $p = |A|$ gives 
\begin{align*}
{\rk}_{\V_{\xi,t}}^*(A) \le \sum_{i \in [|A|]} {\rk}_{\V_{\xi,t}}(A_i)-|A| = 2|A|-|A|=|A|,
\end{align*} because each $A_i$ is a singleton and the subspaces are $2$-dimensional.

To show that we in fact have equality, note that if ${\rk}_{\V_{\xi,t}}^*(A) \ge n$, then by what we just proved it must be that ${\rk}_{\V_{\xi,t}}^*(A) = n$; hence we may assume ${\rk}_{\V_{\xi,t}}^*(A) < n$, that is 
\begin{align*}
\sum_{i \in [p]} {\rk}_{\V_{\xi,t}}(A_i) - p < n,
\end{align*} where $A  = \bigsqcup_{i \in [p]} A_i$ is a partition of $A$ that computes ${\rk}_{\V_{\xi,t}}^*(A)$. In particular $\rk_{\V_{\xi,t}}(A_i) \le n$ for every $i \in [p]$, so from the formula in the first paragraph of the proof, we obtain 
\begin{align*}
{\rk}_{\V_{\xi,t}}^*(A) = \sum_{i \in [p]} {\rk}_{\V_{\xi,t}}(A_i) - p  = \sum_{i \in [p]} 2|A_i| - p=2|A|-p \ge |A|.
\end{align*} As we have already shown ${\rk}_{\V_{\xi,t}}^*(A) \le |A|$, equality must hold.
\end{proof}

\subsection{The product ideal $J_{\xi,t} R(t)$}

Denote by $J_{\xi,t} R(t)$ the ideal of $R(t)$ generated by $J_{\xi,t}$ in $R(t)$, this is the product ideal associated to the subspace arrangement $\V_{\xi,t}$. By $\mf_{R(t)}$ we indicate the maximal homogeneous ideal of $R(t)$.

\begin{lemma} \label{lem:Jc-primary-decomposition}
Suppose $k\ge n$. An irredundant primary decomposition of $J_{\xi,t} R(t)$ is 
\begin{align*}
J_{\xi,t}R(t) =& (\ell_1,x_0+t\ell_1') \cap \cdots \cap (\ell_j,x_0+t\ell_j') \cap \\
&(\ell_{j+1},\ell_{j+1}') \cap \cdots \cap (\ell_{k},\ell_{k}') \cap \mathfrak{m}_{R(t)}^k.
\end{align*}
\end{lemma}
\begin{proof}
As the subspace arrangement $\V_{\xi,t}$ is linearly general, the formula read in the statement is a primary decomposition of $J_{\xi,t}$ by Proposition \ref{prp:primary-decomposition-linearly-general}. As all $(\ell_i,x_0+t\ell_i')$'s for $i \in [j]$ and all $(\ell_i,\ell_i')$'s for $j+1 \le i \le k$ are minimal primes over $J_{\xi,t}R(t)$, it remains to prove that $\mathfrak{m}_{R(t)} \in \Ass(R(t)/J_{\xi,t}R(t))$. As by hypothesis $k \ge n$, and the $n+1$ linear forms $\ell_1,\dots,\ell_n, x_0+t\ell_1'$ are linearly independent over $\EE(t)$ (use the same argument as in the proof of Lemma \ref{lem:rank*-Vc}), $\mathfrak{m}_{R(t)}$ is the sum of all the ideals associated to the subspace arrangement $\V_{\xi,t}$. Hence by Proposition \ref{prp:associated primes}, it suffices to prove 
\begin{align*}
{\rk}_{\V_{\xi,t}}^*([k]) = {\rk}_{\V_{\xi,t}}([k]) -1.
\end{align*} As $k \ge n$, Lemma \ref{lem:rank*-Vc} gives 
\begin{align*}
{\rk}_{\V_{\xi,t}}^*([k]) = \min\{|[k]|,n\} = n,
\end{align*} and the proof is concluded because ${\rk}_{\V_{\xi,t}}([k]) = n+1$.
\end{proof}

\subsection{The subspace arrangement $\V_{\xi,0}$}

Inside the $n+1$ dimensional $\EE$-vector space of linear forms of $R=\EE[x_0,\dots,x_n]$, we define the subspace arrangement
\begin{align*}
\V_{\xi,0}=\big(\Span(x_0,\ell_1),\dots,\Span(x_0,\ell_j),\, \Span(\ell_{j+1},\ell_{j+1}'), \dots, \Span(\ell_k,\ell_k') \big).
\end{align*} 

As $\V_{\xi,0}$ is certainly not linearly general, it is a remarkable fact, and a crucial ingredient in the proof of Theorem \ref{thm:main}, that the rank function $\rk_{\V_{\xi,0}}^*$ coincides with $\rk_{\V_{\xi,c}}^*$. 

\begin{lemma} \label{lem:rank*-V0}
For any $A \subseteq [k]$ we have that $${\rk}_{\V_{\xi,0}}^*(A) = \min\{|A|,n\}.$$
\end{lemma}
\begin{proof}
By a similar consideration as in the proof of Lemma \ref{lem:rank*-Vc}, one observes that
\begin{align*}
{\rk}_{\V_{\xi,0}}(A) = 
\begin{cases}
\min\{2|A|, n+1\}, & \text{if} \, \, \, A_{\le j}=\emptyset \\
\min\{1+|A_{\le j}|+2|A_{>j}|, n+1\}, & \text{if} \, \, \, A_{\le j}\neq \emptyset.
\end{cases}
\end{align*}
The rest of the proof proceeds in analogy to the proof of Lemma \ref{lem:rank*-Vc}. As was done there, selecting the two extremal partitions with one cell and $|A|$ cells, we obtain ${\rk}_{\V_{\xi,0}}^*(A) \le \min\{|A|,n\}.$ 
To prove that that in fact we have equality, let $A = \bigsqcup_{i \in [p]} A_i$ be a partition that computes ${\rk}_{\V_{\xi,0}}^*(A)$. As in the proof of Lemma \ref{lem:rank*-Vc} we may assume ${\rk}_{\V_{\xi,0}}^*(A) < n$, so that in particular, ${\rk}_{\V_{\xi,0}}(A_i) \le n$ for every $i \in [p]$. Hence, from the description of ${\rk}_{\V_{\xi,0}}$ in the beginning of the proof, we have for every $i \in [p]$ that
\begin{align*}
{\rk}_{\V_{\xi,0}}(A_i) \ge 1+|A_i|. 
\end{align*} Consequently, 
\begin{align*}
\sum_{i \in [p]} {\rk}_{\V_{\xi,0}}(A_i)-p \ge \sum_{i \in [p]} (1+|A_i|)-p = \sum_{i \in [p]}|A_i| = |A|,
\end{align*} and the proof of the lemma is concluded.
\end{proof}

\subsection{The product ideal \texorpdfstring{$J_{\xi,0}$}{} }

We define the ideal $J_{\xi,0}$ of $R=\EE[x_0,\dots,x_n]$ as the product ideal associated to the subspace arrangement $\V_{\xi,0}$; concretely
\begin{align*}
J_{\xi,0} = (x_0,\ell_1)\cdots(x_0,\ell_j) (\ell_{j+1},\ell_{j+1}')\cdots(\ell_k,\ell_k').
\end{align*} By $\mf_{R}$ we indicate the maximal homogeneous ideal of $R$.

\begin{lemma} \label{lem:J0-primary-decomposition}
Suppose $k \ge n$. Then an irredundant primary decomposition of $J_{\xi,0}$ is given by the formula 
\begin{align*}
J_{\xi,0} = \left(\bigcap_{c \in [n-1]} \bigcap_{1 \le \alpha_1 <\cdots <\alpha_{c} \le j} (x_0,\ell_{\alpha_1},\dots,\ell_{\alpha_c})^c \right) \cap \left(\bigcap_{j < i \le k}(\ell_i,\ell_i') \right) \cap \mathfrak{m}_R^k.
\end{align*}
\end{lemma}
\begin{proof}
We apply Remark \ref{rem:irredundant-primary-decomposition}. First, by condition (i) of Proposition \ref{prp:associated primes}, we need to identify those $A$'s for which $\rk_{\V_{\xi,0} }(A) < \rk_{\V_{\xi,0} }(B)$ whenever $B \supsetneq A$. The entire ground set $A = [k]$ trivially satisfies this condition, so suppose $A \subsetneq [k]$. If $\rk_{\V_{\xi,0} }(A)$ attains the maximal possible value $n+1$, then the required condition is clearly violated because $\rk_{\V_{\xi,0} }(A)=\rk_{\V_{\xi,0} }([k])$. 
If on the other hand $\rk_{\V_{\xi,0} }(A)< n+1$, then a similar consideration as in the proof of Lemma \ref{lem:rank*-Vc} shows that for any $i \in [k] \setminus A$  
\begin{align*}
{\rk}_{\V_{\xi,0} }(A\cup\{i\}) \ge {\rk}_{\V_{\xi,0} }(A)+1.
\end{align*} In summary, condition (i) of Proposition \ref{prp:associated primes} is satisfied only by $A=[k]$ and any $A \subsetneq [k]$ with $\rk_{\V_{\xi,0} }(A) \le n$; call $\mathscr{A}$ this set of $A$'s. 

Now, among the $A \in \mathscr{A}$ we need to identify the ones that satisfy condition (ii) of Proposition \ref{prp:associated primes}, that is
\begin{align*}
{\rk}_{\V_{\xi,0} }^*(A) = {\rk}_{\V_{\xi,0} }(A)-1.
\end{align*} By Lemma \ref{lem:rank*-V0}, this equality becomes 
\begin{align*}
\min\{|A|,\, n\} = {\rk}_{\V_{\xi,0} }(A)-1.
\end{align*} As $k \ge n$ by hypothesis, the choice $A = [k]$ certainly satisfies the equality, and as $[k] \in \mathscr{A}$, we have $\mathfrak{m} \in \Ass(R/J_{\xi,0} )$. Let $A \in \mathscr{A} \setminus \{[k]\}$. Let us use the description of $\rk_{\V_{\xi,0} }(A)$ given in the beginning of the proof of Lemma \ref{lem:rank*-V0}. If $A_{\le j} = \emptyset$, the formula gives 
\begin{align*}
{\rk}_{\V_{\xi,0} }(A) = \min\{2|A|,\, n+1\},
\end{align*} and coupled with the assumption $\rk_{\V_{\xi,0} }(A) \le n$, implies  that 
\begin{align*}
{\rk}_{\V_{\xi,0} }(A) = 2|A| \le n,
\end{align*} so that in particular $|A| \le n$. These render the above equality  
\begin{align*}
|A| = 2 |A|-1,
\end{align*} that is $|A| = 1$. Hence, the only associated primes of $R/J_{\xi,0}$ corresponding to an $A \in \mathscr{A}$ with $A_{\le j} = \emptyset$ are the $(\ell_i,\ell_i')$'s for every $i >j$. 

If on the other hand $A_{\le j} \neq \emptyset$, the condition becomes 
\begin{align*}
\min\{|A|,\, n\} &= {\rk}_{\V_{\xi,0} }(A)-1 \\
&=\min\{1+|A_{\le j}| + 2|A_{>j}|,n+1\} -1  \\
& = |A|+|A_{>j}|,
\end{align*} where for the third equality we have used that $\rk_{\V_{\xi,0}}(A) \le n$. For the same reason, the right-hand-side of the first equality is at most $n-1$, forcing $|A|<n$. Hence
\begin{align*}
|A| = |A|+|A_{>j}|,
\end{align*} that is $A_{>j} = \emptyset$. We conclude that the only associated primes of $R/J_{\xi,0}$ corresponding to $A = \{\alpha_1, \dots, \alpha_i\} \in \mathscr{A}$ with $A_{\le j} \neq  \emptyset$, are of the form $(x_0,\ell_{\alpha_1},\dots,\ell_{\alpha_i})$ with $\{\alpha_1,\dots,\alpha_i\} \subseteq [j]$. As all $\alpha_s$'s do not exceed $j$, the membership $A \in \mathscr{A}$ is tantamount to saying $|A| \le n-1$. We finally conclude that the associated primes of $R/J_{\xi,0}$ corresponding to $A$ with $A_{\le j} \neq  \emptyset$ are of the form $(x_0,\ell_{\alpha_1},\dots,\ell_{\alpha_i})$ with $ \alpha_s \le j$ for every $s \in [i]$ and $i \le n-1$.
\end{proof}

\subsection{A flat family induced by \texorpdfstring{$J_{\xi,t}$}{J} } \label{subsection:flat-family}

The $\EE$-algebra homomorphism $\EE[t] \rightarrow R[t] / J_{\xi,t}$ induces a projective morphism 
\begin{align*}
\Proj \left(R[t]/J_{\xi,t}\right) \stackrel{\psi}{\rightarrow} \Spec \left(\EE[t]\right).
\end{align*} 

The following is a key lemma in the proof of Theorem \ref{thm:main}.

\begin{lemma} \label{lem:flat-family}
The morphism $\psi$ is flat.
\end{lemma}

We will give two very different proofs of Lemma \ref{lem:flat-family} in \S \ref{subsubsection:flat-family-first-proof}
and \S \ref{subsubsection:flat-family-second-proof}. The first proof relies on combinatorial arguments for subspace arrangements and it will reveal that substituting $t$ in $J_{\xi,t}$ with any $e \in \EE$ gives an ideal that always has the same Betti numbers (which is a stronger statement than merely the ideal always having the same Hilbert polynomial). The second proof relies on establishing a primary decomposition of $J_{\xi,t}$, and it will reveal that the ring homomorphism $\EE[t] \rightarrow R[t] / J_{\xi,t}$ is itself flat (which is a stronger statement than merely $\psi$ being flat). 

With Lemma \ref{lem:flat-family} at hand, we can establish an important fact. 

\begin{lemma} \label{lem:reg-c<=reg-0}
Let $(J_{\xi,t}R(t))^{\sat} \subset R(t)$ and $J_{\xi,0}^{\sat} \subset R$ be the saturations of the ideals $J_{\xi,t}R(t) \subset R(t)$ and $J_{\xi,0} \subset R$. Then 
\begin{align*}
{\reg}_{R(t)}\left((J_{\xi,t}R(t))^{\sat} \right) \le {\reg}_R(J_{\xi,0}^{\sat}).
\end{align*}
\end{lemma}
\begin{proof}
The argument is standard, but we give the details for completeness and clarity.

With $\AA = \Spec(\EE[t])$, let $\I_{\xi,t}$ be the ideal sheaf on $\PP_{\AA}^n$ induced by the homogeneous ideal $J_{\xi,t}$ of $R[t]$. We have a short exact sequence of sheaves on $\PP_{\AA}^n$ 
\begin{align*}
0 \rightarrow \I_{\xi,t} \rightarrow \O_{\PP_{\AA}^n} \rightarrow \O_{\Proj\left(R[t] / J_{\xi,t} \right)} \rightarrow 0.
\end{align*} By Lemma \ref{lem:flat-family}, the $\O_{\PP_{\AA}^n}$-module $\O_{\Proj\left(R[t] / J_{\xi,t} \right)}$ is flat over $\AA$. Since both $\O_{\PP_{\AA}^n}$-modules $\O_{\Proj\left(R[t] / J_{\xi,t} \right)}$ and $\O_{\PP_{\AA}^n}$ are flat over $\AA$, so is $\I_{\xi,t}$ by the behavior of flatness on exact sequences, and hence so is $\I_{\xi,t}(d)$ for any integer $d$, because the twisting functor is exact.

Theorem III.12.8 in \cite{hartshorneAG} applied to the projective morphism $\PP_{\AA}^n \rightarrow \AA$ and to the flat over $\AA$ $\O_{\PP_{\AA}^n}$-module $\I_{\xi,t}(d)$, gives that for each $i \ge 0$ the function
\begin{align*}
u \in \AA \mapsto \dim_{\kappa(u)} H^i \left(\PP_{k(u)}^n, \big(\I_{\xi,t}(d)\big)|_u  \right)
\end{align*} is upper-semicontinuous, where $k(u)$ is the residue class field of $u$. So there is a dense open set $\sU_{i,d} \subseteq \AA$, where the function attains its minimal value. In particular, $\eta \in \sU_{i,d}$, where $\eta$ is the generic point of $\AA$. Denoting by $o \in \AA$ the origin, we have for any $i$ and $d$ that 
\begin{align*}
\dim_{\kappa(\eta)} H^i \left(\PP_{k(\eta)}^n, (\I_{\xi,t})|_\eta \right) \le 
\dim_{\kappa(o)} H^i \left(\PP_{k(o)}^n, (\I_{\xi,t})|_o \right).
\end{align*} 

Now, $(\I_{\xi,t})|_\eta$ is the ideal sheaf on $\PP_{\EE(t)}^n$ defined by the homogeneous ideal $J_{\xi,t} R(t)$ of $R(t)$, and $(\I_{\xi,t})|_o$ is the ideal sheaf on $\PP_{\EE}^n$ defined by the homogeneous ideal $J_{\xi,0}$ of $R$. Also, for every $i \ge 1$ we have a well-known isomorphism of $\EE(t)$-vector spaces
\begin{align*}
H^i \left(\PP_{\EE(t)}^n, (\I_{\xi,t})|_\eta(d) \right) \cong \left[H_{\mf_{R(t)}}^{i+1}\left( J_{\xi,t} R(t) \right) \right]_d,
\end{align*} and similarly an isomorphism of $\EE$-vector spaces
\begin{align*}
H^i \big(\PP_{\EE}^n, (\I_{\xi,t})|_o(d) \big) \cong \left[H_{\mf_{R}}^{i+1}\left( J_{\xi,0} \right) \right]_d.
\end{align*} It follows that for every $d$ and $i \ge 2$ 
\begin{align*}
\dim_{\EE(t)} \left[H_{\mf_{R(t)}}^{i}\left( J_{\xi,t} R(t) \right) \right]_d \le \dim_{\EE} \left[H_{\mf_{R}}^{i}\left( J_{\xi,0} \right) \right]_d,
\end{align*} or equivalently,   
\begin{align*}
\dim_{\EE(t)} \left[H_{\mf_{R(t)}}^{i}\left( \big(J_{\xi,t}R(t)\big)^{\sat} \right) \right]_d \le \dim_{\EE} \left[H_{\mf_{R}}^{i}\left( J_{\xi,0}^{\sat} \right) \right]_d, \, \, \, \, \, \, \forall d, \, \forall i,
\end{align*} because the $0$th and $1$st local cohomology modules of the saturated ideals $\big(J_{\xi,t}R(t)\big)^{\sat}$ and $J_{\xi,0}^{\sat}$ are zero.
\end{proof}

\subsubsection{First proof of Lemma \ref{lem:flat-family}} \label{subsubsection:flat-family-first-proof}
By Theorem III.9.9 in \cite{hartshorneAG}, $\psi$ is flat if and only if every scheme-theoretic fiber 
\begin{align*}
\psi^{-1}(u) = Proj \left(R[t]/J_{\xi,t} \otimes_{\EE[t]} \kappa(u) \right),
\end{align*}
viewed as a closed subscheme of the $n$-dimensional projective space $\PP^n_{\kappa(u)}$ over the residue class field $\kappa(u)$ of $u \in \Spec \left(\EE[t]\right)$, has the same Hilbert polynomial. As $\EE$ is algebraically closed, there are three distinguished types of fibers $\psi^{-1}(u)$ that we need to examine, corresponding to $u$ being 1) the generic point of $\Spec(\EE[t])$, 2) the origin, and 3) an $\EE$-rational point of $\Spec \left(\EE[t]\right)$ other than the origin. 

When $u$ is the generic point of $\Spec(\EE[t])$, $\psi^{-1}(u)$ is the closed subscheme of $\PP_{\EE(t)}^n$ defined by the product ideal $J_{\xi,t} R(t)$ associated to the subspace arrangement $\V_{\xi,t}$. When $u$ is the origin, $\psi^{-1}(u)$ is the closed subscheme of $\PP_{\EE}^n$ defined by the product ideal $J_{\xi,0}$ associated to the subspace arrangement $\V_{\xi,0}$. By Lemma \ref{lem:rank*-Vc} and \ref{lem:rank*-V0} the rank functions $\rk_{\V_{\xi,t}}^*$ and $\rk_{\V_{\xi,0}}^*$ coincide. By Proposition \ref{prp:Betti}, the Betti numbers of $J_{\xi,t} R(t)$ and $J_{\xi,0}$ are fully determined in terms of $\rk_{\V_{\xi,t}}^*$ and $\rk_{\V_{\xi,0}}^*$, respectively. Consequently, the product ideals $J_{\xi,t} R(t)$ and $J_{\xi,0}$ share the same Betti numbers and thus a fortiori the same Hilbert function and Hilbert polynomial.

When $u$ is an $\EE$-rational point of $\Spec(\EE[t])$ other than the origin, say corresponding to the prime ideal $(t-e)$ for some $0 \neq e \in \EE$, then $\psi^{-1}(u)$ is the closed subscheme of $\PP_{\EE}^n$ defined by the ideal of $R$ given by $J_{\xi,e} = (\ell_1,x_0+e \ell_1')\cdots (\ell_j,x_0+e \ell_j')(\ell_{j+1}, \ell_{j+1}')\cdots (\ell_{k}, \ell_{k}').$
This is the product ideal associated to the subspace arrangement
\begin{align*}
\V_{\xi,e} = \big(&\Span(\ell_1,x_0+e\ell_1'),\dots,\Span(\ell_j,x_0+e\ell_j'), \\
&\Span(\ell_{j+1},\ell_{j+1}'), \dots, \Span(\ell_k,\ell_k') \big)
\end{align*} of the vector space of linear forms of $R$. 

Fix any $A \subseteq [k]$, and consider $M_A$, which is an $(n+1) \times 2 |A|$ matrix, with columns ordered so that the first $|A_{\le j} |$ columns contain the coefficients of the forms $x_0 + t \ell_j'$ for $j \in A_{\le j}$, the next $|A|$ columns contain the coefficients of $\ell_j$ for $j \in A$, and the last $|A_{>j}|$ columns contain the coefficients of $\ell_j'$ for $j \in A_{>j}$. We will argue that any two maximal minors of $M_A$, viewed as polynomials in $\EE[t]$, must have distinct non-zero roots. So let $\I_1, \I_2$ and $\J_1, \J_2$ be the row and column indexing sets of these minors; that is, the two minors are $\mu_s(t) :=\det\big((M_A)_{\I_s,\J_s}\big)$ with $s = 1,2$, where $(M_A)_{\I_s,\J_s}$ is the submatrix of $M_A$ associated to rows indexed by $\I_s$ and columns indexed by $\J_s$. 

Suppose first that $0$, which by convention is the index of the top row of $M_{[k]}$, is contained neither in $\I_1$ nor $\I_2$; then the multiplicity of zero as a root of $\mu_s(t)$ coincides with the degree of $\mu_s(t)$, which equals to $|(\J_s)_{\le j}|$. Hence neither $\mu_1(t)$ nor $\mu_2(t)$ has a non-zero root. We may thus assume in the sequel that $0$ is contained in $\I_1 \cap \I_2$. If every element of $\J_s$ is bigger than $j$, then $\mu_s(t)$ is a non-zero constant of $\mathbb{E}[t]$; hence we may further assume that both $(\J_1)_{\le j}$ and $(\J_2)_{\le j}$ are non-empty; under those hypotheses each of $\mu_1(t)$ and $\mu_2(t)$ has a non-zero root, say $\mathfrak{r}_1$ and $\mathfrak{r}_2$, respectively. Let us write $\I_1 = \{0 < i_2 < \dots < i_r \}$ and 
\begin{align*}
\J_1 &= \big\{a_1 <\cdots <a_u < |A_{\le j}|+b_1 < \cdots < |A_{\le j}|+b_v \\
& < |A_{\le j}| + |A| +c_1 < \cdots < |A_{\le j}| + |A| +c_w \big\},
\end{align*} where $u + v + w = r$ and $u$ is positive with $a_u \le |A_{\le j}|$. Then, $(M_A)_{\I_1,\J_1}=$
\begin{align*}
\begin{bmatrix}
1+t \xi_{a_1 0}' & \cdots & 1+t \xi_{a_u 0}' & \xi_{b_1 0} & \cdots & \xi_{b_v 0} & \xi_{c_1 0}' & \cdots & \xi_{c_w 0}' \\
t \xi_{a_1 i_1}' & \cdots & t \xi_{a_u i_1}' & \xi_{b_1 i_1} & \cdots & \xi_{b_v i_1} & \xi_{c_1 i_1}' & \cdots & \xi_{c_w i_1}' \\
\vdots & \cdots & \vdots & \vdots & \cdots & \vdots & \vdots & \cdots & \vdots \\
t \xi_{a_1 i_r}' & \cdots & t \xi_{a_u i_r}' & \xi_{b_1 i_r} & \cdots & \xi_{b_v i_r} & \xi_{c_1 i_r}' & \cdots & \xi_{c_w i_r}'
\end{bmatrix}.
\end{align*} It is easy to see that $\det\big((M_A)_{\I_1,\J_1}\big)$, as a polynomial in $t$, has degree equal to $u$, and zero is a root of multiplicity $u-1$. Moreover, the remaining root $\mathfrak{r}_1$ can be expressed, as an element of $\EE$, by the quotient of two polynomials $\mathfrak{r}_1 = f_1/g_1$, where $g_1$ is a linear form in the $\xi$'s or $\xi'$'s that appear in any column, and $f_1$ is a linear form in the $\xi$'s and $\xi'$'s that appear in any row except the top row.

Now, if $\I_1 \neq \I_2$, then there is some $\alpha \in \I_1 \setminus \I_2$, such that $f_1$ is a linear form in the $\xi_{\alpha, j}$'s and $\xi_{\alpha,j}'$'s with $j \in \J_1$, but $\mathfrak{r}_2$ does not involve any of these variables; hence $\mathfrak{r}_1$ and $\mathfrak{r}_2$ are distinct elements of $\EE$. We may thus assume that $\I_1 = \I_2$, whence necessarily $\J_1 \neq \J_2$. Hence, there exists a $\beta \in \J_1 \setminus \J_2$, such that $g_1$ is a linear form in the $\xi_{i, \beta}$'s or $\xi_{i, \beta}'$'s for $i \in \I_1$, but $\mathfrak{r}_2$ does not involve any of these variables; again, this implies that $\mathfrak{r}_1$ and $\mathfrak{r}_2$ are distinct elements of $\EE$.

Let now $A$ and $A'$ be two distinct subsets of $[k]$. We can always select a maximal minor of $M_A$ and a maximal minor of $M_{A'}$, such that these are distinct maximal minors of $M_{A \cup A'}$. Then by what we proved above, these minors have distinct non-zero roots. From this it follows that, while the subspace arrangement $\V_{\xi,e}$ will not always be linearly general, it may fail the linearly general condition for at most one $A^{\dagger} \subseteq [k]$, for which necessarily $2|A^{\dagger}| = n+1$ (otherwise there are more than one maximal minors involved and only one of them can be zero). Moreover, for $A^{\dagger}$, the vector space dimension of the sum of the subspaces in $\V_{\xi,e}$ indexed by $A^{\dagger}$ will be one less than the maximal possible. Consequently, if $\V_{\xi,e}$ is not linearly general, then 
\begin{align*}
{\rk}_{\V_{\xi,e}}(A)  =
\begin{cases}
\min\{n+1, \, 2|A| \}-1 = 2|A|-1, & A = A^{\dagger} \\
\min\{n+1, \, 2|A| \}, & A \neq A^{\dagger} 
\end{cases}.
\end{align*} Let us now show that $\rk_{\V_{\xi,e}}^*$ agrees with $\rk_{\V_{\xi,t}}^*$. By the same argument as in the proof of Lemma \ref{lem:rank*-Vc}, we have $\rk_{\V_{\xi,e}}^*(A) \le \min\{|A|,n\}$, and we may assume that $\rk_{\V_{\xi,e}}^*(A) < n$, so that $\rk_{\V_{\xi,e}}(A_i) \le n$ for every $i \in [p]$, where 
\begin{align*}
{\rk}_{\V_{\xi,e}}^*(A) = \sum_{i \in [p]} {\rk}_{\V_{\xi,e}}(A_i) - p. 
\end{align*} If $A_i \neq A^{\dagger}$ for every $i \in [p]$, then the right-hand-side is bounded from below by $|A|$, for the same reason as in the proof of Lemma \ref{lem:rank*-Vc}, so that ${\rk}_{\V_{\xi,e}}^*(A) = |A|$ and the claimed equality $\rk_{\V_{\xi,e}}^*(A) = \min\{|A|,n\}$ holds. Hence suppose $A_1 = A^\dagger$. Then
\begin{align*}
{\rk}_{\V_{\xi,e}}(A_1)-1 = 2|A_1|-2 \ge |A_1|, 
\end{align*} because necessarily $|A^{\dagger}| \ge 2$. In summary,
\begin{align*}
{\rk}_{\V_{\xi,e}}^*(A) & = \sum_{i \in [p]} {\rk}_{\V_{\xi,e}}(A_i) - p, \\
&= (2|A_1|-1-1) + \sum_{2 \le i \le p} (2|A_i|-1) \\
&\ge |A_1| + \sum_{2 \le i \le p} |A_i| = |A|.
\end{align*} Hence $\rk_{\V_{\xi,e}}^*(A) = \min \{|A|, \, n\} = \rk_{\V_{\xi,t}}^*(A)$ as well. 

\subsubsection{Second proof of Lemma \ref{lem:flat-family}} \label{subsubsection:flat-family-second-proof}

We will prove that the $\EE$-algebra homomorphism $\phi: \EE[t] \rightarrow R[t] / J_{\xi,t}$ is flat, from which the statement will follow immediately. As flatness is preserved by base change, we may assume that $\EE$ is the field of rational functions in the $\xi_{pq}$'s and $\xi_{pq}'$'s over $\KK$ (i.e., we do not need to assume $\EE$ is algebraically closed). As $\EE[t]$ is a principal ideal domain, the flatness of $\phi$ is equivalent to $R[t] / J_{\xi,t}$ being torsion-free as $\EE[t]$-module. This latter property will follow after explicitly describing an irredundant primary decomposition of $J_{\xi,t}$, which we do in the rest of this proof.

Recalling the earlier defined linear forms for $s \in [k]$
\begin{align*}
\ell_s &= \xi_{0s} x_0 + \cdots + \xi_{ns} x_n, \, \, \, \\
\ell_s' &= \xi_{0s}' x_0 + \cdots + \xi_{ns}' x_n, 
\end{align*} and setting $L_s  = (\ell_s, \, x_0 + t \ell_s')$, we consider the product and intersection of the $L_s$'s for $s \in [j]$:
\begin{align*}
J &= L_1 \cdots L_j, \\
I &= L_1 \cap \cdots \cap L_j
\end{align*}

We first establish a primary decomposition of $J$ when $j <n$.

\begin{proposition} \label{prp:J-cap=prod}
Suppose $n \ge 2$. If $j < n$, then $J = I$.
\end{proposition}
\begin{proof}
We induct on $n$. The base of the induction is $n=2$, where necessarily $j=1$ and the statement is trivially true. For $n \ge 3$, we induct on $j$, where for $j=1$ the statement is again trivially true. So suppose $n \ge 3$ and $j>1$. Set
\begin{align*}
J_j := L_1 \cdots L_{j-1}.
\end{align*} We claim $\ell_j, \, x_0+t \ell_j'$ is an $R[t]/J_j$-sequence. By the induction hypothesis 
\begin{align*}
J_j = L_1 \cap  \cdots \cap L_{j-1}.
\end{align*} As the $L_s$'s are prime ideals, they are the associated primes of $J_j$, and $\ell_j$ is $R[t]/J_j$-regular if and only if $\ell_j$ does not belong to any $L_s$. For the sake of a contradiction, suppose $\ell_j \in L_s$ for some $s \in [j-1]$. Then 
\begin{align*}
\sum_{i=0}^n \xi_{ij}x_i= c \left(\sum_{i=0}^n \xi_{is}x_i \right) + c'\big[(1+t\xi_{0s}')x_0+t\xi_{s1}'x_1+\cdots+t\xi_{ns}'x_n\big]
\end{align*} for some $c, \, c' \in R[t]$. As the left-hand-side is a linear form, we may assume that $c, \, c' \in \EE[t]$. As $n \ge 3$, isolating the coefficients of $x_{1}$ and $x_2$ in the above relation, yields the equations
\begin{align*}
\xi_{1j} &= c \xi_{1s} + c't\xi_{1s}', \\
\xi_{2j} &= c \xi_{2s} + c't\xi_{2s}'. 
\end{align*} Writing $c = c_0 + c_1 t$ and $c' = c_0' + c_1't$ with $c_0, \, c_0' \in \EE$ and $c_1, \, c_1' \in \EE[t]$, we extract the two relations
\begin{align*}
\xi_{1j} = c_0 \xi_{1s}, \, \, \, \, \, \, \xi_{2j} = c_0 \xi_{2s}.
\end{align*} These imply that $\xi_{1j} / \xi_{1s} = \xi_{2j} / \xi_{2s}$, which is a contradiction, because the $\xi_{pq}$'s are algebraically independent over $\KK$. Hence $\ell_j$ is indeed $R[t] / L_j$-regular.

We next show that $x_0+t\ell_j'$ is $R[t] / (L_j+(\ell_j))$-regular.
Set $\overline{R}[t] = R[t] / (\ell_j)$ and denote by $\overline{J}_j$ the image of $J_j$ in $\overline{R}[t]$. Thinking 
of $\overline{R}[t]$ as the polynomial ring over $\EE[t]$ in the variables $x_0,\dots,x_{n-1}$, obtained from $R[t]$ via the substitution
\begin{align*}
x_n \mapsto - \xi_{nj}^{-1} (\xi_{0j} x_0 + \cdots + \xi_{(n-1) j} x_{n-1}),
\end{align*} we see that the image $\overline{L}_s$ of $L_s$ in $\overline{R}[t]$ is generated by the polynomials
\begin{align*}
\mu_s &:= \zeta_{0s} x_0 + \cdots + \zeta_{(n-1)s} x_{n-1}, \\
\mu_s' &:= x_0 + t \big(\zeta_{0s}'x_0 + \cdots + \zeta_{(n-1)s}'x_{n-1} \big), 
\end{align*} with $\zeta_{is}, \, \zeta_{is}' \in \EE$ defined as 
\begin{align*}
\zeta_{is} & = \xi_{is}-\xi_{ns}\xi_{ij}\xi_{nj}^{-1}, \\
\zeta_{is}' & = \xi_{is}'-\xi_{ns}'\xi_{ij}\xi_{nj}^{-1}, 
\end{align*} for $0 \le i \le n-1, \, s \in [j-1]$. Note that the $\zeta_{is}$'s and $\zeta_{is}'$'s are algebraically independent over $\KK$, and denote by $\EE'$ the field they generate over $\KK$ (this is a subfield of $\EE$). 
With $\overline{R}' = \EE'[x_0,\dots,x_{n-1}]$, define for $s \in [j-1]$ ideals of $\overline{R}'[t]$ as $\overline{L}_s' = (\mu_s, \, \mu_s')$. As $n \ge 3$ by assumption, $n-1 \ge 2$ and the induction hypothesis on $n$ gives 
\begin{align*}
\overline{J}' := \overline{L}_1' \cdots \overline{L}_{j-1}' = \overline{L}_1' \cap \cdots \cap \overline{L}_{j-1}'.
\end{align*} By further setting
\begin{align*}
\zeta_{ns} & = \xi_{ns}, \, \, \, s \in [j-1], \\
\zeta_{ns}' & = \xi_{ns}', \, \, \, s \in [j-1], \\
\zeta_{ij} & = \xi_{ij}, \, \, \, 0 \le i \le n, \\
\zeta_{ij}' & = \xi_{ij}', \, \, \, 0 \le i \le n,
\end{align*} it is clear that the collection of all $\zeta_{is}$'s and $\zeta_{is}'$'s are algebraically independent over $\KK$ and they generate $\EE$. By construction, we have that $\overline{L}_s' \overline{R}[t] = \overline{L}_s$, where $\overline{L}_s' \overline{R}[t]$ is the ideal generated by $\overline{L}_s'$ in $\overline{R}[t]$. Then  
\begin{align*}
\overline{J}_j & = \overline{L}_1 \cdots \overline{L}_{j-1} \\
& =  (\overline{L}_1'\overline{R}[t]) \cdots (\overline{L}_{j-1}'\overline{R}[t]) \\
& = (\overline{L}_1' \cdots \overline{L}_{j-1}')\overline{R}[t] \\
& = (\overline{L}_1' \cap \cdots \cap \overline{L}_{j-1}')\overline{R}[t] \\
& = (\overline{L}_1'\overline{R}[t]) \cap \cdots \cap (\overline{L}_{j-1}'\overline{R}[t]) \\
& = \overline{L}_1 \cap \cdots \cap \overline{L}_{j-1},
\end{align*} where for the fifth equality we used that the ring extension $\overline{R}'[t] \hookrightarrow \overline{R}[t]$, induced by the field extension $\EE' \hookrightarrow \EE$, is flat. Identifying $R[t] / (J_j + (\ell_j))$ with $\overline{R}[t] / \overline{J}_j$, we have by what we just proved that $x_0 + t \ell_j'$ is $R[t] / (J_j + (\ell_j))$-regular if and only if the class of $x_0 + t \ell_j'$ in $\overline{R}[t]$ does not belong to any of the $\overline{L}_s$'s for $s \in [j-1]$. Thus for the sake of a contradiction, suppose that $\overline{x_0 + t \ell_j'} \in \overline{L}_s$ for some $s \in [j-1]$. This implies that 
\begin{align*}
x_0 + t \left(\sum_{i=0}^{n-1} \zeta_{ij}'x_i \right) = c \left( \sum_{i=0}^{n-1} \zeta_{is}x_i \right) + c' \left[x_0 + t \left(\sum_{i=0}^{n-1} \zeta_{is}'x_i \right)  \right],
\end{align*} for some $c, \, c' \in \overline{R}[t]$, where the $\zeta_{ij}'$'s for $0 \le i \le n-1$ are now defined as 
\begin{align*}
\zeta_{ij}' := \xi_{ij}' - \xi_{nj}^{-1} \xi_{ij} \xi_{nj}'.
\end{align*} As the left-hand-side has degree $1$, we may assume that $c, \, c' \in \EE[t]$. As before, let us write $c = c_0 + t c_1$ and $c' = c_0' + t c_1'$ with $c_0, \, c_0' \in \EE$ and $c_1, \, c_1' \in \EE[t]$. Isolating in both sides of the above relation the component of the coefficient of $x_1$ that lies in $\EE$ (which is zero in the left-hand-side), we see that $c_0 = 0$. Doing the same for the component of the coefficient of $x_0$ that lies in $\EE$ and using that $c_0 =0$, we obtain $c_0' = 1$. As $n-1 \ge 2$, we further isolate the component 
of the coefficients of $x_1$ and $x_2$ that is divisible by $t$, to respectively obtain 
\begin{align*}
\zeta_{1j}' &= c_1 \zeta_{1s} +  \zeta_{1s}' + tc_1'\zeta_{1s}', \\
\zeta_{2j}' &= c_1 \zeta_{2s} + \zeta_{2s}' + tc_1'\zeta_{2s}'.
\end{align*} Eliminating $c_1$, obtain that $tc_1' \in \EE$. This is however a contradiction, as $t c_1'$ is transcendental over $\EE$.

We have shown that $L_j$ is generated by an $R[t]/J_j$-sequence. Hence, 
\begin{align*}
\frac{J_j \cap L_j}{J_j L_j} = {\Tor}_1^{R[t]}\left(\frac{R[t]}{J_j}, \frac{R[t]}{L_j} \right) = 0.
\end{align*} Consequently, 
\begin{align*}
J = J_j L_j = J_j \cap L_j  = \big(L_1 \cap \cdots \cap L_{j-1} \big) \cap L_j,
\end{align*} where the first equality is by definition, the second equality is by the vanishing of Tor that we just showed, and the third equality is by the induction hypothesis on $j$.
\end{proof}

With the aid of Proposition \ref{prp:J-cap=prod} we next establish a primary decomposition of $J$ when $j \ge n$. Denote by $\mf_{R[t]}$ the ideal of $R[t]$ generated by $x_0,\dots,x_n$. 

\begin{proposition} \label{prp:Jj-j>=n}
For any $n \ge 0$ and $j \ge n$, we have $J = I \cap \mf_{R[t]}^j$. 
\end{proposition}
\begin{proof}
When $n=0$, every $L_s$ is the ideal generated by $x_0$, and the statement holds trivially. For $n \ge 1$, we induct on $j \ge n$, by adapting as needed the ingenious idea of the proof of Lemma 3.2 in \cite{conca2003castelnuovo}. For $n=1$ the base of the induction is for $j=1$, which is trivial. Hence to complete the proof, we need to establish the base of the induction ($j=n$) for $n \ge 2$ and the induction step ($j \ge n+1$) for $n \ge 1$. We will treat these simultaneously. For any $s \in [j]$ set 
\begin{align*}
J_{\hat{s}} = \prod_{i \in [j] \setminus \{s\}} L_i.
\end{align*}  In view of Proposition \ref{prp:J-cap=prod} or the induction hypothesis on $j$, we have the respective primary decomposition
\begin{align*}
J_{\hat{s}} = \begin{cases}
\cap_{i \neq s} L_i, & j = n \ge 2 \\
\left(\cap_{ i \neq s} L_i \right) \cap \mf_{R[t]}^{j-1}, & j \ge n+1.
\end{cases}
\end{align*} Regardless, it will suffice to prove (the crucial inclusion is $ \supseteq$)
\begin{align*}
J = J_{\hat{1}} \cap \cdots \cap J_{\hat{j}} \cap \mf_{R[t]}^j. 
\end{align*}  

Consider the linear form
\begin{align*}
x = \ell_1 + \cdots + \ell_j.
\end{align*} We claim that $x$ does not belong to any $L_s$. If not, then 
\begin{align*}
\sum_{i=0}^n \left(\sum_{f=1}^j \xi_{if} \right) x_i = c \left(\sum_{i=0}^n \xi_{is} x_i \right) + c' \left[x_0 + t \left(\sum_{i=0}^n \xi_{is}' x_i \right) \right]
\end{align*} for some $s \in [j]$, where $c, \, c' \in \EE[t]$. Write $c = c_0+tc_1$ and $c' = c_0' + tc_1'$ with $c_0, \, c_0' \in \EE$ and $c_1, \, c_1' \in \EE[t]$. Then from the above relation we extract the equations
\begin{align*}
\xi_{01} + \cdots + \xi_{0j} & = c_0 \xi_{0s} + c_0', \\
\xi_{11} + \cdots + \xi_{1j} & = c_0 \xi_{1s}, \\ 
0 & = c_1 \xi_{0s} + c_0' \xi_{0s}' + tc_1' \xi_{0s}' +c_1', \\
0 & = c_1 \xi_{1s} + c_0' \xi_{1s}' + tc_1' \xi_{1s}'. 
\end{align*} Multiplying the third equation with $\xi_{1s}$, the fourth with $\xi_{0s}$ and subtracting, yields
\begin{align*}
0 = c_0'(\xi_{0s}'\xi_{1s}-\xi_{1s}'\xi_{0s}) + tc_1'(\xi_{0s}'\xi_{1s}-\xi_{1s}'\xi_{0s})+c_1' \xi_{1s}.
\end{align*} Suppose $c_1' \neq 0$ and write $u$ for its leading term. Viewing the right-hand-side as a polynomial in $t$, and since $c_0' \in \EE$, we see that the leading term of the right-hand-side is $tu(\xi_{0s}'\xi_{1s}-\xi_{1s}'\xi_{0s}) \neq 0$. As the left-hand-side is zero, this is a contradiction, so that $c_1'$ must be zero. But then $c_0'$ must be zero as well. However, a new contradiction now arises  from the first two equations, on the account that $j \ge 2$. This concludes the proof that $x \not\in L_s$ for any $s \in [j]$. 

Let us use the bar notation to indicate modulo $x$. The induction hypothesis on $n$ gives the equality of ideals
\begin{align*}
\overline{J} = \overline{J}_{\hat{1}} \cap \cdots \cap \overline{J}_{\hat{j}} \cap \overline{\mf}_{\overline{R}[t]}^j 
\end{align*} in the ring $\overline{R}[t] = R[t]/(x)$ (one applies the induction hypothesis over $\EE'[t][x_0,\dots,x_{n-1}]$ where $\EE'$ is a subfield of $\EE$ and then extends the decomposition to $R[t]$, as was done in the proof of Proposition \ref{prp:J-cap=prod}). 

Let $f$ be a homogeneous element of $R[t]$ of degree at least $j$ that belongs to every $J_{\hat{s}}$; that is $f \in J_{\hat{1}} \cap \cdots \cap J_{\hat{j}} \cap \mf_{R[t]}^j$. By the equality in the previous paragraph we have $\overline{f} \in \overline{J}$, so that $f = g + x f'$, with $g \in J$ and $f'$ homogeneous of degree at least $j-1$. It follows that $x f' \in J_{\hat{s}}$ for every $s \in [j]$. As $x \not\in L_s$ for any $s$ and the degree of $f'$ is at least $j-1$, the aforementioned primary decomposition of $J_{\hat{s}}$ furnishes $f' \in J_{\hat{s}}$ for every $s \in [j]$. As $\ell_s \in L_s$, this yields $\ell_s f' \in J$, so $x f' \in J$, and thus $f \in J$. 
\end{proof} 

Set $L_s = (\ell_s, \ell_s')$ for every $j+1 \le s \le k$; then by definition
\begin{align*}
J_{\xi,t} = L_1 \cdots L_k.
\end{align*} Recall from the induction hypothesis in the proof of Theorem \ref{thm:main} (\S \ref{subsection:induction-hypothesis}) that $j \ge 2n-2$ and $n \ge 3$; in particular $j \ge n+1$. Hence the proof of Lemma \ref{lem:flat-family} is finally concluded in view of the next proposition.

\begin{proposition}
Suppose $j \ge n+1$. Then $$J_{\xi,t} = L_1 \cap \cdots \cap L_k \cap \mf_{R[t]}^k.$$
\end{proposition}
\begin{proof}
We induct on $n \ge 0$. For $n=0$ every $L_i$ is the ideal generated by $x_0$ and the statement is trivial. For $n \ge 1$, we induct on $k-j \ge 0$. The base of the induction is $k-j = 0$, which is taken care of by Proposition \ref{prp:Jj-j>=n}. For $k-j>0$ we apply the same method as in the proof of Proposition \ref{prp:Jj-j>=n} (which is an adaptation of the method in the proof of Lemma 3.2 in \cite{conca2003castelnuovo}). In particular, setting for any $j+1 \le s \le k$ 
\begin{align*}
J_{\xi,t,\widehat{s}} = L_1 \cdots L_{s-1} L_{s+1} \cdots L_k,
\end{align*} the induction hypothesis on $k-j$ furnishes
\begin{align*}
J_{\xi,t,\widehat{s}} = L_1 \cap \cdots \cap L_{s-1}  \cap L_{s+1} \cap \cdots \cap L_k \cap  \mf_{R[t]}^{k-1}.
\end{align*} Hence it suffices to prove that (the crucial inclusion is $\supseteq$)
\begin{align*}
J_{\xi,t} = J_{\xi,t,\widehat{j+1}} \cap \cdots \cap J_{\xi,t,\widehat{k}} \cap \mf_{R[t]}^k.
\end{align*} Consider the linear from
\begin{align*}
x = \ell_1 + \cdots + \ell_k.
\end{align*} A similar argument as in the proof of Proposition \ref{prp:Jj-j>=n} shows that $x \not\in L_s$ for any $s \in [k]$. Let $f \in R[t]$ be homogeneous of degree at least $k$, such that $f \in J_{\xi,t,\hat{s}}$ for every $j+1 \le s \le k$. Going modulo $x$, the induction hypothesis on $n$ gives $f = g + xf'$ with $g \in J_{\xi,t}$ and $f'$ homogeneous of degree at least $k-1$. As $J_{\xi,t} \subseteq J_{\xi,t,\hat{s}}$, we have $xf' \in J_{\xi,t,\hat{s}}$, hence $f' \in J_{\xi,t,\hat{s}}$, hence $xf' \in J_{\xi,t}$, and finally $f \in J_{\xi,t}$.
\end{proof}

\subsection{A star configuration and its hyperplane section} \label{subsection:star-configuration-hyperplane-section}

With $c \le n-1$, we consider the ideals of $R$
\begin{align*}
A_{\xi,c,j} &= \bigcap_{1 \le \alpha_1 < \cdots < \alpha_c \le j} (\ell_{\alpha_1},\dots,\ell_{\alpha_c}), \, \, \, \text{and} \\
B_{\xi,c,j} &= \bigcap_{1 \le \alpha_1 < \cdots < \alpha_c \le j} (x_0,\ell_{\alpha_1},\dots,\ell_{\alpha_c}).
\end{align*} The ideals $A_{\xi,c,j}$ and $B_{\xi,c,j}$ define in $\PP_{\EE}^n$ and $\PP_{\EE}^{n-1} \cong \Proj\big(R/(x_0)\big)$ respectively, a codimension $c$ star configuration of $j$ hyperplanes. To relate the two ideals, we need the following folklore criterion. 

\begin{lemma} \label{lem:comparison-I-J}
Let $J \subseteq I$ be two homogeneous ideals of $R$ such that 
\begin{enumerate}[label = (\roman*)]
\item $J$ and $I$ have the same height,
\item $R/J$ and $R/I$ have the same multiplicity, and 
\item for any $P \in \Ass(R/J)$, we have $\height(P) = \height(J)$. 
\end{enumerate}Then $J = I$. 
\end{lemma}
\begin{proof}
The proof is straightforward once one sets to argue that $J$ and $I$ must have the same primary decomposition.
\end{proof}

\begin{lemma} \label{lem:star-configuration-hyperplane-section}
$B_{\xi,c,j} = A_{\xi,c,j} + (x_0)$.
\end{lemma}
\begin{proof}
Noting that $A_{\xi,c,j} + (x_0) \subseteq B_{\xi,c,j}$, we apply the criterion of Lemma \ref{lem:comparison-I-J}. It is clear that 
\begin{align*}
\height \big(A_{\xi,c,j} + (x_0)\big) = \height(B_{\xi,c,j}) = c+1.
\end{align*} Next, both $R/A_{\xi,c,j}$ and $R/B_{\xi,c,j}$ have the same multiplicity, because they have the same $h$-vector by Proposition \ref{prp:star-configurations}(iv). Thus the first two conditions of Lemma \ref{lem:comparison-I-J} are satisfied. For the third condition, note that $x_0$ is $R/A_{\xi,c,j}$-regular because $c \le n$. As by Proposition \ref{prp:star-configurations}(i) $R/A_{\xi,c,j}$ is Cohen-Macaulay, $R/\big(A_{\xi,c,j}+(x_0)\big)$ will also be Cohen-Macaulay. In particular, for any $P \in \Ass \left(R/\big(A_{\xi,c,j}+(x_0)\big) \right)$, we will have $\height(P) = c+1$. Thus the third condition of Lemma \ref{lem:comparison-I-J} is also satisfied.
\end{proof}

\subsection{Returning to $S$}

By Lemma \ref{lem:Jc-primary-decomposition}, the saturation of the ideal $J_{\xi,t} R(t)$ is the ideal of $R(t)$ given by
\begin{align*}
I_{\xi,t}  :=(\ell_1,x_0+t \ell_1') \cap \cdots \cap (\ell_j,x_0+t \ell_j') \cap (\ell_{j+1},\ell_{j+1}) \cap \cdots \cap (\ell_k, \ell_k'). 
\end{align*} This is the saturated ideal that defines in the projective space $\PP_{\EE(t)}^n$ the reduced union of the linear spaces 
\begin{align*}
&\Proj\big(R(t) / (\ell_i,x_0+t \ell_i'\big), \, i \in [j], \, \, \, \text{and} \\
&\Proj\big(R(t) / (\ell_i,\ell_i')\big), \, i \in [k] \setminus [j].
\end{align*} Similarly, by Lemma \ref{lem:Jc-primary-decomposition}, the saturation of $J_{\xi,0}$ is the ideal of $R$ given by 
\begin{align*}
I_{\xi,0} &= \left(\bigcap_{c \in [n-1]} \bigcap_{1 \le \alpha_1 <\cdots <\alpha_{c} \le j} (x_0,\ell_{\alpha_1},\dots,\ell_{\alpha_c})^c \right) \cap \left(\bigcap_{j < i \le k}(\ell_i,\ell_i') \right).
\end{align*} Recalling the definition of the ideal $B_{\xi,c,j}$ from \S \ref{subsection:star-configuration-hyperplane-section}, we have  
\begin{align*}
I_{\xi,0} = \left(\bigcap_{c \in [n-1]} B_{\xi,c,j}^{(c)} \right) \cap \left(\bigcap_{j < i \le k}(\ell_i,\ell_i') \right),
\end{align*} where $B_{\xi,c,j}^{(c)}$ denotes the $c$th symbolic power. Thus geometrically, $I_{\xi,0}$ defines in $\PP_{\EE}^n$ the reduced union of the $k-j$ codimension-$2$ linear spaces
\begin{align*}
\Proj\big(R / (\ell_i,\ell_i')\big), \, i \in [k] \setminus [j],
\end{align*} together with the $c$th symbolic power, for all $c \in [n-1]$, of the codimension-$c$ star configuration associated to the $j$ codimension-$1$ planes
\begin{align*}
\Proj \big(R / (x_0, \ell_i) \big), \, i \in [j]
\end{align*} inside the hyperplane defined by $x_0$. 

Now, quite generally, the Castelnuovo-Mumford regularity of a polynomially parametrized homogeneous ideal (such as $I_{\xi,t}$ and $I_{\xi,0}$; the parameters here being $\xi$) coincides with the regularity of the ideal obtained by substituting the parameters with generic values over the infinite ground field. The argument for this is standard, using i) the Buchsbaum-Eisenbud acyclicity criterion for complexes of finite free modules, ii) the fact that all graded components of the homology modules in the specialized dual graded minimal free resolution achieve their minimal vector space dimension on a dense open subset of the parameter space, and iii) graded local duality. Consequently, 
\begin{align*}
{\reg}_{R(t)} \big(I_{\xi,t}\big) = {\reg}_S (I_X).
\end{align*} Moreover, $\reg_{R}(I_{\xi,0})$ coincides with the regularity of the ideal of $S$
\begin{align*}
Q &= \left(\bigcap_{c \in [n-1]} \bigcap_{1 \le \alpha_1 <\cdots <\alpha_{c} \le j} (x_0,f_{\alpha_1},\dots,f_{\alpha_c})^c \right) \cap \left(\bigcap_{j < i \le k}(f_i,f_i') \right)
\end{align*} where $f_1,\dots,f_j, \, f_1', \dots,f_k'$ are generic linear forms of $S$. 

It follows from Lemma \ref{lem:reg-c<=reg-0} that 
\begin{align*}
{\reg}_S (I_X) \le {\reg}_S(Q). 
\end{align*} The remaining of the proof is devoted to bounding the regularity of $Q$.

\subsection{Traces of residuals}

We develop the machinery that will allow us to express the ideal
\begin{align*}
\left(\overline{Q:x_0^c}\right)^{\sat}: = \left(\frac{Q:x_0^c+(x_0)}{x_0}\right)^{\sat}
\end{align*} in a particularly convenient form. We recall that the saturation index $\sat(J)$ of a homogeneous ideal $J$ of $S$ is the smallest integer $\alpha$ such that $[J]_d  = [J^{\sat}]_d$ for every $d \ge \alpha$.

\begin{lemma} \label{lem:JKLc}
Set $L_c = (f_1,\dots,f_c)$. Then for any two homogeneous ideals $J$ and $K$ of $S$, we have 
\begin{align*}
(J \cap L_c + K)^{\sat} &= (J+K)^{\sat} \cap (L_c+K)^{\sat}, \, \, \, \text{and} \\
\sat(J \cap L_c +K) &\le \max\left\{\sat(K+L_c), \, \reg(J+K)+1 \right\}.
\end{align*}
\end{lemma}
\begin{proof}
Tensoring the short exact sequence 
\begin{align*}
0 \rightarrow \frac{S}{J \cap L_c} \rightarrow \frac{S}{J} \oplus \frac{S}{L_c} \stackrel{\phi} \rightarrow \frac{S}{J+L_c} \rightarrow 0
\end{align*} with $S/K$ over $S$, we obtain the following exact sequnce: 

\hspace{-0.5cm}
\begin{tikzcd}
{\Tor}_1^S\left(\frac{S}{J}\oplus \frac{S}{L_c},\frac{S}{K} \right)  \arrow[r, "\phi'"] \arrow[d, equal] & {\Tor}_1^S\left(\frac{S}{J+L_c},\frac{S}{K} \right) \arrow[d, equal] \arrow[r] & \frac{S}{J \cap L_c+K} \arrow[r] & \frac{S}{J+K} \oplus \frac{S}{L_c+K} \arrow[d] \\
\frac{J\cap K}{JK} \oplus \frac{L_c \cap K}{L_cK} \arrow[r,"\tilde{\sigma}"] &  \frac{(J+L_c)\cap K}{(J+L_c)K} & 0 & \arrow[l] \frac{S}{J+L_c+K}
\end{tikzcd}

As the map $\phi$ is induced by the inclusion of ideals $J$ and $L_c$ into the ideal $L_c+J$, so is the map $\phi'$, which can be identified with a map $\tilde{\sigma}$ as shown in the commutative diagram. We consider the composite map
\begin{align*}
\sigma: \frac{J\cap K}{JK} \hookrightarrow \frac{J\cap K}{JK} \oplus \frac{L_c \cap K}{L_cK} \stackrel{\tilde{\sigma}} \rightarrow \frac{(J+L_c)\cap K}{(J+L_c)K},
\end{align*} where the first arrow is the canonical inclusion. Arguing inductively, we can factor $\sigma$ as 
\begin{align*}
\sigma: \frac{J\cap K}{JK} \stackrel{\sigma_1}{\longrightarrow} \frac{(J+L_1)\cap K}{(J+L_1)K} \stackrel{\sigma_2} {\longrightarrow} \frac{(J+L_2)\cap K}{(J+L_2)K} \stackrel{\sigma_3}{\longrightarrow} \cdots \stackrel{\sigma_c}{\longrightarrow} \frac{(J+L_c)\cap K}{(J+L_c)K},
\end{align*} where $L_i = (f_1,\dots,f_i)$ for every $ i \in [c]$.

We now claim that for every 
\begin{align*}
d \ge \max \left\{\sat(J+K), \, \sat(J+L_1+K), \dots, \sat(J+L_{c-1}+K) \right\} +1
\end{align*} all $\sigma_i$'s are surjective on the degree-$d$ components. For this it will suffice to prove that for every $d \ge \sat(J + L_{i-1}+K)+1$, $\sigma_i$ is surjective on the degree-$d$ components. So let $g_i \in (J+L_i)\cap K$ be homogeneous of degree $d$. Then we can write $g_i = g_{i-1} + f_i h_i$, where $g_{i-1} \in J+L_{i-1}$ and $h_i$ is homogeneous of degree $d-1$. As $g_i \in K$, it follows that $f_i h_i \in J+L_{i-1}+K$. Since the degree of $h_i$ is at least $\sat(J+L_{i-1}+K)$, and since, being generic, $f_i$ is almost regular on $S/(J+L_{i-1}+K)$ (that is, $f_i$ is a non-zero divisor of $S/(J+L_{i-1}+K)^{\sat}$), we infer $h_i \in J+L_{i-1}+K$. So we can write $h_i = h_{i-1} + h_i'$, where $h_{i-1} \in J+L_{i-1}$ and $h_i' \in K$. Substituting in the original expression for $g_i$, we obtain 
\begin{align*}
g_i =g_{i-1} + f_i h_{i-1}+f_i h_i'.
\end{align*} As both $g_{i-1}$ and $h_{i-1}$ belong to $J+L_{i-1}$, we have 
\begin{align*}
g_{i-1}' := g_{i-1} + f_i h_{i-1} \in J+L_{i-1}. 
\end{align*} As both $g_i$ and $h_i'$ belong to $K$, we in fact have 
\begin{align*}
g_{i-1}' \in (J+L_{i-1}) \cap K. 
\end{align*}
Finally, as $f_i h_i' \in L_i K \subseteq (J+L_i)K$, we have  
\begin{align*}
g_i + (J+L_i)K = g_{i-1} + f_i h_{i-1}+ f_i h_i'+(J+L_i)K = g_{i-1}' + (J+L_i)K.
\end{align*} This proves that $\sigma_i$ is surjective, and hence that $\sigma$ is surjective at all degrees in the claim. 

As the $f_i$'s are generic, applying $i$ times Corollary 1.3 in \cite{conca2003castelnuovo}, gives that for any $i$ 
\begin{align*}
\sat(J+L_i+K) \le \reg(J+K). 
\end{align*} Hence $\sigma$ is surjective on the degree-$d$ components for $d \ge \reg(J+K)+1$. From the commutative diagram, it follows that for such $d$'s
\begin{align*}
\left[J\cap L_c +K\right]_d &= \left[(J+K) \cap (L_c+K) \right]_d. 
\end{align*} That is, eventually the ideal $J \cap L_c +K$ agrees with the ideal $(J+K) \cap (L_c+K)$, which proves that
\begin{align*}
(J\cap L_c +K)^{\sat} = \big((J+K) \cap (L_c+K) \big)^{\sat} = (J+K)^{\sat} \cap (L_c+K)^{\sat}. 
\end{align*} With that, the stated upper bound on $\sat(J \cap L_c+K)$ is clear. 
\end{proof}

Recalling the definition of $Q$ as 
\begin{align*}
Q &= \left(\bigcap_{c \in [n-1]} \bigcap_{1 \le \alpha_1 <\cdots <\alpha_{c} \le j} (x_0,f_{\alpha_1},\dots,f_{\alpha_c})^c \right) \cap \left(\bigcap_{j < i \le k}(f_i,f_i') \right), 
\end{align*} we further write 
\begin{align*}
Q &= Q' \cap \left(\bigcap_{j < i \le k}(f_i,f_i') \right), \\
Q' & := \bigcap_{c \in [n-1]} B_{c,j}^{(c)}, \\
B_{c,j} & := \bigcap_{1 \le \alpha_1 <\cdots <\alpha_{c} \le j} (x_0,f_{\alpha_1},\dots,f_{\alpha_c}), \\
A_{c,j} & := \bigcap_{1 \le \alpha_1 <\cdots <\alpha_{c} \le j} (f_{\alpha_1},\dots,f_{\alpha_c}).
\end{align*} 

\begin{lemma} \label{lem:A-i-th-colon-bar}
For any $c \in \{0, \, 1, \, \dots, \, n-2\}$ we have 
\begin{align*}
A_{c+1,j}+(x_0)  = Q' : x_0^c + (x_0) = B_{c+1,j}.
\end{align*}
\end{lemma}
\begin{proof}
We first argue that we have inclusions
\begin{align*}
A_{c+1,j}+(x_0)  \subseteq Q' : x_0^c + (x_0) \subseteq B_{c+1,j}.
\end{align*} Since the $f_i$'s are generic linear forms, for any $c' \in [n-1]$ and any $\{\alpha_1, \dots, \alpha_{c'}\} \subseteq [n-1]$, the polynomials $x_0,f_{\alpha_1},\dots,f_{\alpha_{c'}}$ are a regular sequence, hence 
\begin{align*}
(x_0,f_{\alpha_1},\dots,f_{\alpha_{c'}})^{c'} : x_0^c = 
\begin{cases}
(x_0,f_{\alpha_1},\dots,f_{\alpha_{c'}})^{c'-c}, & c' > c \\
S, & c' \le c.
\end{cases}
\end{align*} Consequently, the second inclusion follows from the fact that
\begin{align*}
Q' : x_0^c &= \left(\bigcap_{c' \in [n-1]} B_{c',j}^{(c')}\right) : x_0^c 
 = \bigcap_{c' \in [n-1]} \left(B_{c',j}^{(c')} : x_0^c\right) \\
 &= \bigcap_{c' \in [n-1] \setminus [c]} B_{c',j}^{(c'-c)} \subseteq B_{c+1,j}. 
\end{align*} To see the first inclusion, recall that by Proposition \ref{prp:star-configurations}(iii), the ideal $A_{c+1,j}$ is generated by all products of the form $f_{\beta_1} \cdots f_{\beta_{j-c}}$ for $1 \le \beta_1 < \cdots < \beta_{j-c} \le j$. Now, for every $c' \in [n-1] \setminus [c]$, the set 
\begin{align*}
\{\beta_1, \dots, \beta_{j-c}\} \cap \{\alpha_1, \dots, \alpha_{c'}\}
\end{align*} has cardinality at least $c'-c$, implying that 
\begin{align*}
f_{\beta_1} \cdots f_{\beta_{j-c}} \in (x_0,f_{\alpha_1},\dots,f_{\alpha_{c'}})^{c'-c}.
\end{align*} Expanding the above formula for $Q':x_0^c$, we get 
\begin{align*}
Q':x_0^c = \bigcap_{c' \in [n-1] \setminus [c]} \bigcap_{1 \le \alpha_1 <\cdots <\alpha_{c'} \le j} (x_0,f_{\alpha_1},\dots,f_{\alpha_{c'}})^{c'-c},
\end{align*} which indeed reveals the first inclusion. Finally, an identical argument as in the proof of Lemma \ref{lem:star-configuration-hyperplane-section} gives
\begin{align*}
A_{c+1,j}+(x_0)  = B_{c+1,j},
\end{align*} forcing the inclusions in the beginning of the proof to be equalities.
\end{proof}

To state our formula for $(\overline{Q:x_0^c})^{\sat}$, for any ideal $I$ of $S$ we define $\overline{I}  = \big(I+(x_0)\big) / (x_0)$. Moreover, we set $I_i = (f_i,f_i')$ for $j+1 \le i \le k$. 

\begin{lemma} \label{lem:trace of residual}
For any $c \in \{0, \, 1, \, \dots, \, n-2\}$, we have
\begin{align*}
(\overline{Q:x_0^c})^{\sat} = 
\overline{B}_{c+1,j} \cap \overline{I_{j+1}} \cap \cdots \cap \overline{I_{k}}.\end{align*} 
\end{lemma}
\begin{proof} 
Apply Lemma \ref{lem:JKLc} with $J = Q':x_0^c, \, K = (x_0)$ and $L_c = I_{j+1}$ to get 
\begin{align*}
\big(Q':x_0^c +(x_0)\big)^{\sat} \cap \big(I_{j+1}+(x_0)\big) = \big((Q':x_0^c)\cap I_{j+1}+(x_0) \big)^{\sat}.
\end{align*} Apply Lemma \ref{lem:JKLc} again with $J = \big((Q':x_0^c)\cap I_{j+1}+(x_0) \big)^{\sat}, \, K=(x_0)$ and $L_c = I_{j+2}$ to get 
\begin{align*}
&\big(Q':x_0^c +(x_0)\big)^{\sat} \cap \big(I_{j+1}+(x_0)\big) \cap \big(I_{j+2}+(x_0) \big) = \\
&\big((Q':x_0^c)\cap I_{j+1}+(x_0) \big)^{\sat} \cap \big(I_{j+2}+(x_0) \big) = \\
&\big((Q':x_0^c)\cap I_{j+1} \cap I_{j+2}+(x_0) \big)^{\sat}.
\end{align*} Continuing inductively, we obtain 
\begin{align*}
(\overline{Q:x_0^c})^{\sat} = (\overline{Q':x_0^c})^{\sat} \cap \overline{I_{j+1}} \cap \cdots \cap \overline{I_{k}}.
\end{align*} By Lemma \ref{lem:A-i-th-colon-bar}, we see that 
\begin{align*}
\overline{Q':x_0^c} = \frac{Q':x_0^c +(x_0)}{(x_0)} = \overline{B}_{c+1,j}.
\end{align*} and this concludes the proof because the ideal $\overline{B}_{c+1,j}$ is saturated. 
\end{proof}

\subsection{Bounding the regularity of traces of residuals} 
We establish the following bound on the regularity $(\overline{Q:x_0^c})^{\sat}$:

\begin{lemma} \label{lem:reg-Q}
For any $c \in \{0, \, 1, \, \dots, \, n-2\}$, we have 
\begin{align*}
\reg \left((\overline{Q:x_0^c})^{\sat}\right) \le (n-1)j-c.
\end{align*}
\end{lemma} 
\begin{proof}
By Lemma \ref{lem:trace of residual}, 
\begin{align*}
(\overline{Q:x_0^c})^{\sat} = 
\overline{B}_{c+1,j} \cap \overline{I_{j+1}} \cap \cdots \cap \overline{I_{k}}.\end{align*} 
As the ideal $\cap_{j+1 \le i \le k} \overline{I}_i$ is in general coordinates with respect to the ideal $\overline{B}_{c+1,j}$, Proposition 2.9 in \cite{caviglia2016hilbert} asserts that the module 
\begin{align*}
{\Tor}_1^S\left(\frac{\overline{S}}{\overline{B}_{c+1,j}}, \frac{\overline{S}}{\cap_{j+1 \le i \le k} \overline{I}_i} \right)
\end{align*} has finite length. Hence Corollary 1.3 in \cite{EHU} yields
\begin{align*}
&\reg \left(\frac{\overline{S}}{\overline{B}_{c+1,j} + \cap_{j+1 \le i \le k} \overline{I}_i} \right) = \reg \left( {\Tor}_0^S\left(\frac{\overline{S}}{\overline{B}_{c+1,j}}, \frac{\overline{S}}{\cap_{j+1 \le i \le k} \overline{I}_i} \right) \right) \\
& \le \reg\left(\frac{\overline{S}}{\overline{B}_{c+1,j}}\right) + \reg \left( \frac{\overline{S}}{\cap_{j+1 \le i \le k} \overline{I}_i} \right) \\
& = 
\reg\left(\overline{B}_{c+1,j}\right) + \reg \left( \cap_{j+1 \le i \le k} \overline{I}_i \right)-2.
\end{align*} Hence, the regularity lemma and the short exact sequence
\begin{align*}
0 \rightarrow \frac{\overline{S}}{(\overline{Q:x_0^c})^{\sat}} \rightarrow \frac{\overline{S}}{\overline{B}_{c+1,j}} \oplus \frac{\overline{S}}{\cap_{j+1 \le i \le k} \overline{I}_i} \rightarrow \frac{\overline{S}}{\overline{B}_{c+1,j} + \cap_{j+1 \le i \le k} \overline{I}_i} \rightarrow 0,
\end{align*} together with the above inequality, give 
\begin{align*}
&\reg \left((\overline{Q:x_0^c})^{\sat}\right) \le  \\
&\max \left\{ 
\reg (\overline{B}_{c+1,j}), \, \reg (\cap_{j+1 \le i \le k} \overline{I}_i), \, \reg(\overline{B}_{c+1,j} + \cap_{j+1 \le i \le k} \overline{I}_i)+1)  
\right\} \le \\
&\reg\left(\overline{B}_{c+1,j}\right) + \reg \left( \cap_{j+1 \le i \le k} \overline{I}_i \right).
\end{align*} By Proposition \ref{prp:star-configurations}(ii)
\begin{align*}
\reg \left( \overline{B}_{c+1,j}\right) = j-c.
\end{align*} As $k-j ={j +1 \choose 2} - { j \choose 1} = {j \choose 2}$, the running induction hypothesis on $n$ gives
\begin{align*}
\reg \left(\cap_{j+1 \le i \le k} \overline{I}_i \right) \le (n-2)(j-1).
\end{align*} Now we are done, since 
\begin{align*}
\reg \left((\overline{Q:x_0^c})^{\sat}\right) & \le \reg\left(\overline{B}_{c+1,j}\right) + \reg \left( \cap_{j+1 \le i \le k} \overline{I}_i \right) \\
& = j-c + (n-2)(j-1) = (n-1)j-c-n+2 \\
&\le (n-1)j-c, 
\end{align*} where the last inequality is by our assumption that $n \ge 2$.
\end{proof}

\subsection{Bounding the regularity of $Q$}

We finish the proof of Theorem \ref{thm:main} by proving:

\begin{lemma}
$\reg(Q) \le (n-1) j$.
\end{lemma}
\begin{proof}
For any $c \in [n-2]$ we have a short exact sequence of graded $S$-modules
\begin{align*}
0 \rightarrow Q: x_0^c(-1) \stackrel{x_0}{\rightarrow} Q: x_0^{c-1} \rightarrow \overline{Q: x_0^{c-1}} \rightarrow 0.
\end{align*} 
This induces an exact sequence on graded components of local cohomology modules
\begin{align*}
\left[H_{\mf}^i \big(Q: x_0^c \big) \right]_{d-1} \rightarrow \left[H_{\mf}^i \big( Q: x_0^{c-1} \big) \right]_{d} \rightarrow \left[H_{\mf}^i \big(\overline{Q: x_0^{c-1}} \big)\right]_d.
\end{align*} As $Q$ is saturated, so is $Q:x_0^{c-1}$, so that 
\begin{align*}
\reg \big(Q:x_0^{c-1} \big) \le&  \max_{i \ge 2} \left\{a_i \big(Q: x_0^c \big)+i+1, \, a_i \big(\overline{Q: x_0^{c-1}} \big)+i \right\}= \\
&  \max_{i \ge 2} \left\{a_i \big(Q: x_0^c \big)+i+1, \, a_i \Big(\big(\overline{Q: x_0^{c-1}} \big)^{\sat}\Big)+i \right\},
\end{align*} where for a finite graded $S$-module $M$ we define
\begin{align*}
a_i(M) = \max \left\{ d: \, [H_{\mf}^i(M)]_d \neq 0 \right\}.
\end{align*} As a consequence, for every $c \in [n-1]$, we have 
\begin{align*}
\reg \big(Q: x_0^{c-1}\big) \le \max \left\{\reg \big(Q: x_0^c \big)+1, \reg \Big( \big( \overline{Q: x_0^{c-1}}\big)^{\sat}\Big) \right\}.
\end{align*} Combine these $n-1$ inequalities to get
\begin{align*}
\reg(Q) \le \max \left\{\reg \big( Q:x_0^{n-1}\big)+n-1, \max_{0 \le c \le n-2}\left[  \reg \Big(\big( \overline{Q: x_0^c}\big)^{\sat}\Big)+c\right]\right\}.
\end{align*} Now observe that 
\begin{align*}
Q:x_0^{n-1} = I_{j+1} \cap \cdots \cap I_k
\end{align*} is the ideal of $k-j = {j \choose 2}$ generic codimension-$2$ linear spaces in $\PP_{\KK}^n$, so that by the running induction hypothesis on $j$, we have
\begin{align*}
\reg \big(Q:x_0^{n-1} \big) \le (n-1)(j-1). 
\end{align*} Moreover, Lemma \ref{lem:reg-Q} asserts that for all $c \in \{0, \, 1, \, \dots, \, n-2\}$
\begin{align*}
\reg \left((\overline{Q:x_0^c})^{\sat}\right) \le (n-1)j-c.
\end{align*} Putting everything together, we arrive at 
\begin{align*}
\reg(Q) &\le \max \left\{(n-1)(j-1)+n-1, \max_{0 \le c \le n-2}\left[  (n-1)j-c+c\right]\right\} \\
& =(n-1)j,
\end{align*} concluding the proof of the lemma.
\end{proof}

\section{Proof of Proposition \ref{prp:lower-bound}} \label{section:lower-bound}

\subsection{Points in $\PP^2$ and lines in $\PP^3$}
The Hilbert function of the homogeneous coordinate ring $S/I_X$ of $k$ generic points in $\PP^2$ is well-known to be 
\begin{align*}
\HF(S/I_X,d) =  \min \left\{{d+2 \choose 2}, k \right\}.
\end{align*} Thus the initial degree of $I_X$ is the smallest positive integer $d$ such that 
\begin{align*}
{d+2 \choose 2} - k > 0.
\end{align*} The two roots of the quadratic on the left-hand-side are 
\begin{align*}
d_{\pm} = \frac{-3 \pm \sqrt{1+8k}}{2}.
\end{align*} As $d_{-} \le 0$, the smallest positive integer for which the quadratic is positive is indeed $\lfloor d_{+} \rfloor + 1$. The formula 
\begin{align*}
\lfloor (-5+\sqrt{1+24k})/2 \rfloor+1, 
\end{align*} for the initial degree of $I_X$ when $X$ is the union of $k$ generic lines in $\PP^3$, is derived in exactly the same way, using instead the theorem of Hartshorne \& Hirschowitz \cite{hartshorne1982droites}, which asserts that
\begin{align*}
\HF(S/I_X,d) =  \min \left\{{d+3 \choose 3}, k(d+1) \right\}.
\end{align*} 

\subsection{An induction lemma} \label{subsection:induction-lemma}

Let $I_{k,n}$ denote the ideal of $k$ generic codimension-2 linear spaces in $\PP^n$. The following lemma is an inductive device for bounding from below the initial degree of $I_{k,n}$. 

\begin{lemma} \label{lem:inductive-initial-degree}
Suppose there is an integer $1 \le j \le k-1$ such that 
$\HF(I_{k-j,n},d-1) = \HF(I_{k-j,n-1},d-j) = 0$. Then $\HF(I_{k,n},d)=0$.
\end{lemma}
\begin{proof}
Let $\ell_i, \, \ell_i', \, i \in [k]$, be generic linear forms of $S$. With $t$ a transcendental over $S$, consider the extended polynomial ring $S[t]$, equipped with a grading that extends the grading of $S$ by assigning degree zero to $t$. Consider the ideal of $S[t]$ given by 
\begin{align*}
J_t = (\ell_1, x_0 +t \ell_1') \cap \cdots \cap (\ell_j, x_0 +t \ell_j') \cap (\ell_{j+1},\ell_{j+1}') \cap \cdots \cap (\ell_k,\ell_k').
\end{align*} As all associated primes of $S[t] / J_t$ contract to the zero ideal in $\KK[t]$, we have that $S[t] / J_t$ is flat over $\KK[t]$. Consequently, by Exercise 20.14 in \cite{Eisenbud:CA}, the Hilbert function of the fiber ring $S /J_c := S[t] / \big(J_t +(t-c)\big)$ is constant for any $c \in \KK$. As $I_X$ is isomorphic to $J_c$ for a generic choice of $c$, it suffices to prove $\HF(J_0,d)=0$. As we have an inclusion of ideals of $S  = S[t]/(t)$ 
\begin{align*}
J_0 := \frac{J_t + (t)}{(t)} \hookrightarrow (\ell_1, x_0 ) \cap \cdots \cap (\ell_j, x_0) \cap (\ell_{j+1},\ell_{j+1}') \cap \cdots \cap (\ell_k,\ell_k'):= J_0',
\end{align*} it will further suffice to prove that $\HF(J_0',d)=0$. We will achieve this by showing that the two terms in the right-hand-side of the following inequality are zero: 
\begin{align*}
\HF(J_0',d) \le \HF(J_0':x_0,d-1) + \HF\left(\left(\frac{J_0'+(x_0)}{(x_0)}\right)^{\sat},d \right)
\end{align*} As $J_0':x_0 = I_{k-j,n}$, the first term vanishes by the hypothesis. To control the second term, notice that
\begin{align*}
\left(\frac{J_0'+(x_0)}{(x_0)}\right)^{\sat} \hookrightarrow (\overline{\ell}_1\cdots \overline{\ell}_j) \cap (\overline{\ell}_{j+1},\overline{\ell}_{j+1}') \cap \cdots \cap (\overline{\ell}_k,\overline{\ell}_k') \cong I_{k-j,n-1}(-j),
\end{align*} where the overbar notation indicates modulo $x_0$, and $I_{k-j,n-1}(-j)$ is the ideal $I_{k-j,n-1}$ shifted in degree by $-j$. It thus suffices to prove that $\HF(I_{k-j,n-1},d-j) = 0$, which is true by hypothesis.
\end{proof}

\subsection{Planes in $\PP^4$} We next bound from below the initial degree of the ideal of planes in $\PP^4$.

\begin{lemma}
Let $X$ be the union of $k$ generic $2$-dimensional linear spaces in $\PP^4$. Then the initial degree of $I_X$ is at least $3 \sqrt{k} -4$.
\end{lemma}
\begin{proof}
We will show that for any $k, d$ satisfying 
\begin{align*}
9 k \ge (d+5)^2,
\end{align*} the conditions of Lemma \ref{lem:inductive-initial-degree} are met. We do this by induction on $d \ge 1$. When $d=1$, there are at least $4$ planes, which are certainly not contained in any hyperplane. 

Suppose $d \ge 2$ and set $k=\lceil\frac{(d+5)^2}{9} \rceil$ and $k'=\lceil \frac{(d+4)^2}{9} \rceil$. The induction hypothesis asserts that $\HF(I_{k',4},d-1) = 0$, where $I_{k',4}$ is the ideal of the union of $k'$ generic planes in $\PP^4$. To finish, we need to prove that $\HF(I_{k',3},d-k+k') = 0$, where $I_{k',3}$ is the ideal of the union of $k'$ generic lines in $\PP^3$. For this, it is enough, by Hartshorne \& Hirschowitz, to show 
\begin{align*}
{d-k+k'+3 \choose 3} \le k' (d-k+k'+1). \tag{$*$}
\end{align*} In turn, this will follow once we establish the numerical inequality
\begin{align*}
6 \, \frac{(d+4)^2}{9} \geq \left[ d-\frac{(d+5)^2}{9}+\frac{(d+4)^2}{9}+1+3 \right]^2 \tag{$**$},
\end{align*} for then the outermost of the following inequalities is $(*)$:
\begin{align*}
6 k' \geq 6 \, \frac{(d+4)^2}{9} & \stackrel{(**)}{\geq} \left[ d-\frac{(d+5)^2}{9}+\frac{(d+4)^2}{9}+1+3 \right]^2 \\
& \geq (d-k+k'+3)^2 \\
& \ge (d-k+k'+3)(d-k+k'+2)
\end{align*} To see that $(**)$ is true, take the square-root of both sides to obtain
\begin{align*}
\frac{\sqrt{6}}{3} d+\frac{4\sqrt{6}}{3} \geq \frac{7}{9}d+3,
\end{align*} and note that $\sqrt{6}/3 > 7/9$ and $4 \sqrt{6}/3 > 3$.
\end{proof}

\subsection{Ascension of vanishing}
When $X$ is the union of $k$ generic codimension $2$ linear spaces in $\PP^n$ with $n \ge 4$, a lower bound on the initial degree of $I_{X}$ follows as the maximum of the corresponding lower bounds for points in $\PP^2$, lines in $\PP^3$ and planes in $\PP^4$, via \emph{ascension of vanishing}: 

\begin{lemma} \label{lem:ascencion of vanishing degree}
Let $k'$ be a positive integer and $n \ge 2$. If the ideal $I_{k',n-1}$ of the union of $k'$ generic codimension-$2$ linear spaces in $\PP_{\KK}^{n-1}$ is zero at degree $d$, then so is the ideal $I_{k',n}$ of the union of $k'$ generic codimension-$2$ linear spaces in $\PP_{\KK}^{n}$.
\end{lemma}
\begin{proof}
Writing $I_{k',n} = \cap_{i \in [k']} I_i$, where $I_i$ is an ideal of $S$ generated by two generic linear forms, we may assume that 
\begin{align*}
I_{k',n-1} = \frac{\cap_{i \in [k']} \big(I_i+(\ell)\big)}{(\ell)},
\end{align*} where $\ell$ is a generic linear form. We have the short exact sequence of $\KK$-vector spaces
\begin{align*}
0 \rightarrow \left[ I_{k',n}\right]_{d-1} \rightarrow [I_{k',n}]_d \rightarrow \left[\frac{ I_{k',n} +(\ell)}{(\ell)}\right]_{d} \rightarrow 0.
\end{align*} From the embedding
\begin{align*}
\left[\frac{ I_{k',n} +(\ell)}{(\ell)}\right]_{d} \hookrightarrow \left[ I_{k',n-1}\right]_{d} 
\end{align*} and the hypothesis that $\left[ I_{k',n-1}\right]_{d} =0$, we get 
\begin{align*}
\left[ I_{k',n}\right]_{d-1} \cong \left[ I_{k',n}\right]_{d}.
\end{align*} But as $I_{k',n-1}$ vanishes at degree $d$, it must vanish at any degree $\nu \le d$, so for every $\nu \le d$ we have
\begin{align*}
\left[ I_{k',n}\right]_{\nu-1} \cong \left[ I_{k',n}\right]_{\nu}.
\end{align*} It follows that $[I_{k',n}]_d = 0$.
\end{proof}

\bibliography{regularity_sublinear}

\end{document}